\documentclass[10pt]{article}
\usepackage{amsfonts}
\usepackage[dvips]{graphicx}
\usepackage{amsmath, amssymb, amscd}
\DeclareFontEncoding{OT2}{}{}\newcommand{\textcyr}[1]{{\fontencoding{OT2}\fontfamily{wncyr}\fontseries{m}\fontshape{n}
     \selectfont #1}}

\newcommand{\ilim}[1]{\displaystyle{\lim_{\genfrac{}{}{0pt}{}{\longleftarrow}{\scriptstyle #1}} }\;}

\newcommand{\Sha}{{\mbox{\textcyr{Sh}}}}
\newcommand{\Zha}{{\mbox{\textcyr{Zh}}}}

\newcommand{\Gal}{{\operatorname{Gal}}}
\newcommand{\rk}{{\operatorname{rk}}}

\newcommand{\cyc}{{\operatorname{cyc}}}
\newcommand{\coker}{{\operatorname{coker}}}

\newcommand{\Pic}{{\operatorname{Pic}}}

\newcommand{\cd}{{\operatorname{cd}}}
\newcommand{\tors}{{\operatorname{tors}}}

\newcommand{\im}{{\operatorname{im}}}

\newcommand{\Hom}{{\operatorname{Hom}}}
\newcommand{\Sel}{{\operatorname{Sel}}}

\newcommand{\Tr}{{\operatorname{Tr}}}

\newtheorem{theorem}{Theorem}[section]
\newtheorem{lemma}[theorem]{Lemma}
\newtheorem{proposition}[theorem]{Proposition}
\newtheorem{hypothesis}[theorem]{Hypothesis}
\newtheorem{conjecture}[theorem]{Conjecture}
\newtheorem{corollary}[theorem]{Corollary}
\newenvironment{proof}[1][Proof]{\begin{trivlist}
\item[\hskip \labelsep {\bfseries #1}]}{\end{trivlist}}
\newenvironment{definition}[1][Definition]{\begin{trivlist}
\item[\hskip \labelsep {\bfseries #1}]}{\end{trivlist}}
\newenvironment{example}[1][Example]{\begin{trivlist}
\item[\hskip \labelsep {\bfseries #1}]}{\end{trivlist}}
\newenvironment{remark}[1][Remark]{\begin{trivlist}
\item[\hskip \labelsep {\bfseries #1}]}{\end{trivlist}}

\newcommand{\qed}{\nobreak \ifvmode \relax \else
      \ifdim\lastskip<1.5em \hskip-\lastskip
      \hskip1.5em plus0em minus0.5em \fi \nobreak
      \vrule height0.75em width0.5em depth0.25em\fi}

\begin{document}

\title{Some remarks on the two-variable main conjecture of Iwasawa theory for elliptic curves without complex multiplication}
\author{Jeanine Van Order \footnote{The author acknowledges support from the Swiss National 
Science Foundation (FNS) grant 200021-125291.}}
\date{}

\maketitle
\vspace{-1cm}

\begin{abstract}
We establish several results towards the two-variable main conjecture of Iwasawa theory for elliptic curves without complex multiplication 
over imaginary quadratic fields, namely (i) the existence of an appropriate $p$-adic $L$-function, building on works of Hida and Perrin-Riou, 
(ii) the basic structure theory of the dual Selmer group, following works of Coates, Hachimori-Venjakob, et al., and (iii) the 
implications of dihedral or anticyclotomic main conjectures with basechange. The result of (i) is deduced from the construction of Hida and 
Perrin-Riou, which in particular is seen to give a bounded distribution. The result of (ii) allows us to deduce a corank 
formula for the $p$-primary part of the Tate-Shafarevich group of an elliptic curve in the ${\bf{Z}}_p^2$-extension of 
an imaginary quadratic field. Finally, (iii) allows us to deduce a criterion for one divisibility of the two-variable main conjecture 
in terms of specializations to cyclotomic characters, following a suggestion of Greenberg, as well as a refinement 
via basechange. \end{abstract}

\section{Introduction}

Fix a prime $p \in {\bf{Z}}$. Given a profinite group $G$, let $\Lambda(G)$ denote
the ${\bf{Z}}_p$-Iwasawa algebra of $G$, which is the completed group ring 
\begin{align*}\Lambda(G) = {\bf{Z}}_p[[G]] = \varprojlim_{U} {\bf{Z}}_p[G/U].\end{align*}
Here, the projective limit runs over all open normal subgroups $U$ of $G$. Note that the elements of $\Lambda(G)$
can be viewed in a natural way as ${\bf{Z}}_p$-valued measures on $G$. Let $E$ be an elliptic curve defined over ${\bf{Q}}$ 
of conductor $N$. Hence $E$ is modular by fundamental work of Wiles \cite{Wi}, Taylor-Wiles \cite{TW}, and Breuil-Conrad-
Diamond-Taylor \cite{BCDT}, with Hasse-Weil $L$-function $L(E,s)$ given by that of a cuspidal newform $f \in S_2(\Gamma_0(N))$. 

Let $k$ be an imaginary quadratic field. The Hasse-Weil $L$-function $L(E/k, s)$ of $E$ over $k$ is given by the Rankin-Selberg 
$L$-function $L(f \times \Theta_k, s)$, where $\Theta_k$ is the theta series associated to $k$ by a classical construction (as described 
for instance in \cite{GZ}). Let $k_{\infty}$ denote the compositum of all ${\bf{Z}}_p$-extensions of $k$, which by class field theory is
a ${\bf{Z}}_p^2$-extension. Let $G$ denote the Galois group $\Gal(k_{\infty}/k)$. The complex conjugation automorphism of $\Gal(k/{\bf{Q}})$ 
acts on $G$ with eigenvalues $\pm 1$. Let $k^{\cyc}$ denote the ${\bf{Z}}_p$-extension associated to the $+1$-eigenspace, which is the cyclotomic 
${\bf{Z}}_p$-extension of $k$. Let $D_{\infty}$ denote the ${\bf{Z}}_p$-extension associated to the $-1$-eigenspace, which is the dihedral or anticyclotomic 
${\bf{Z}}_p$-extension of $k$. Let $\Gamma$ denote the cyclotomic Galois group $\Gal(k^{\cyc}/k)$, and let $\Omega$ denote the dihedral or anticyclotomic 
Galois group $\Gal(D_{\infty}/k)$. Let $H$ denote the Galois group $\Gal(k_{\infty}/k^{\cyc})$, which is naturally isomorphic to $\Omega \cong {\bf{Z}}_p$.
Let $X(E/k_{\infty})$ denote the Pontryagin dual of the $p^{\infty}$-Selmer group of $E$ over $k_{\infty}$, which has the natural structure of 
a compact $\Lambda(G)$-module. The subject of this note is the following conjecture, made in the spirit of Iwasawa (but often attributed to Greenberg and
Mazur), known as the {\it{two-variable main conjecture of Iwasawa theory for elliptic curves}}:

\begin{conjecture}\label{2vmc}
Let $E$ be an elliptic curve defined over ${\bf{Q}}$, and $p$ a prime where $E$ has either good ordinary 
or multiplicative reduction. 

\begin{itemize}

\item[(i)] There is a unique element $L_p(E/k_{\infty}) \in \Lambda(G)$ that interpolates $p$-adically the central values 
$L(E/k, \mathcal{W}, 1)/ \Omega_f$. Here, $L(E/k, \mathcal{W}, s)$ is the Hasse-Weil $L$-function of $E$ over $k$ twisted
by a finite order character $\mathcal{W}$ of $G$, and $\Omega_f$ is a suitable complex period for which the quotient 
$L(E/k, \mathcal{W}, 1)/ \Omega_f$ lies in $\overline{{\bf{Q}}}$ (and hence in $\overline{{\bf{Q}}}_p$ via any fixed embedding 
$\overline{{\bf{Q}}} \rightarrow \overline{{\bf{Q}}}_p$).

\item[(ii)] The dual Selmer group $X(E/k_{\infty})$ is $\Lambda(G)$-torsion, hence has a $\Lambda(G)$-characteristic power 
series $\operatorname{char}_{\Lambda(G)}X(E/k_{\infty})$.

\item[(iii)] The equality of ideals $\left( L_p(E/k_{\infty}) \right) = \left( \operatorname{char}_{\Lambda(G)}X(E/k_{\infty}) \right)$ holds in $\Lambda(G).$
\end{itemize}

\end{conjecture} In the setting where $E$ has complex multiplication by $k$, much is known about this conjecture thanks to work of 
Rubin \cite{Ru} (see also \cite{Ru2}), building on previous work of Coates-Wiles \cite{CW} and Yager \cite{Ya}. Here, we consider
the somewhat more mysterious setting where $E$ does {\it{not}} have complex multiplication, and in particular what can be deduced from known
Iwasawa theoretic results for the one-variable cases corresponding to the Galois groups $\Gamma$ and $\Omega$.
 
We start with the construction of $p$-adic $L$-functions, (i).  
Given a finite order character $\mathcal{W}$ of $G$, let $\mathcal{W}\left(\lambda\right)$ denote the specialization to 
$\mathcal{W}$ of an element $ \lambda \in \Lambda(G)$. That is, writing $d\lambda$ to denote the measure associated to $\lambda$, 
let $$\mathcal{W}\left(\lambda\right) = \int_G \mathcal{W}(g)\cdot
d\lambda(g).$$ Fix a cuspidal Hecke eigenform $f \in S_2(\Gamma_0(N))$
of weight $2$, level $N$, and trivial Nebentypus. Such an eigenform $f \in S_2(\Gamma_0(N))$ is
said to be {\it{$p$-ordinary}} if its $T_p$-eigenvalue is a $p$-adic
unit with respect to any embedding $\overline{\bf{Q}}
\rightarrow \overline{{\bf{Q}}}_p.$ Let \begin{align*} \langle f, f \rangle_N 
= \int_{\Gamma_0(N) \backslash \mathfrak{H}} \vert f \vert^2 dx dy \end{align*} denote the Petersson
inner product of $f $ with itself. Let $L(f \times \Theta(\mathcal{W}), s)$ denote the
Rankin-Selberg $L$-function of $f$ times the theta series $ \Theta(\mathcal{W})$ associated
to $\mathcal{W}$, normalized to have central value at $s=1$. The
ratio \begin{align*} \frac{L(f \times \Theta(\mathcal{W}), 1)}{8 \pi^2 \langle f ,f \rangle_N} \end{align*} 
lies in $\overline{\bf{Q}}$ by an important theorem of Shimura \cite{Sh1}. Using this fact, along with
the constructions of Hida \cite{Hi} and Perrin-Riou \cite{PR0}, we obtain the following result.

\begin{theorem}[Theorem  \ref{2vinterpolation}]\label{2vplfn}  

Fix an embedding $\overline{\bf{Q}} \rightarrow \overline{\bf{Q}}_p.$ Let $f \in S_2(\Gamma_0(N))$ 
be a $p$-ordinary eigenform of weight $2$, level $N$, and trivial Nebentypus. Assume that $N$ is 
prime to the discriminant of $k$, and that $p \geq 5$. There exists an element $\mu_f
\in \Lambda(G)$ whose specialization to any finite order character $\mathcal{W}$ of 
$G$ satisfies the interpolation formula \begin{align*}\mathcal{W}\left( \mu_f \right) = \eta\cdot
\frac{L(f \times \Theta(\overline{\mathcal{W}}), 1)}{8 \pi^2 \langle f ,f
\rangle_N} \in \overline{{\bf{Q}}}_p, \end{align*} where $\eta = \eta(f, \mathcal{W})$ is a certain 
explicit (nonvanishing) algebraic number.
\end{theorem} Hence, we obtain a $p$-adic $L$-function $L_p(E/k_{\infty}) = L_p(f/k_{\infty}) \in \Lambda(G)$
associated to this measure $\mu_f$.

\begin{remark}The two-variable $p$-adic $L$-function $L_p(f/k_{\infty})$ 
corresponding to $d\mu_f$ also satisfies a functional equation,
as described in Corollary \ref{FE} below. \end{remark} We now consider
the Iwasawa module structure theory of (ii), using standard techniques. Recall
that we let $H$ denote the Galois group $\Gal(k_{\infty}/k^{\cyc})$, 
which is naturally isomorphic to the dihedral or anticyclotomic Galois 
group $\Omega \cong {\bf{Z}}_p$. If $E$ has good ordinary reduction at 
$p$, then an important theorem of Kato \cite{KK} with a nonvanishing theorem of Rohrlich \cite{Ro}
implies that the dual Selmer group $X(E/k^{\cyc})$ is $\Lambda(\Gamma)$-torsion. To be more
precise, the construction of Kato \cite{KK} with the nonvanishing theorem of Rohlich \cite{Ro} show 
that the dual Selmer group $X(E/{\bf{Q}}^{\cyc})$ is $\Lambda(\Gal({\bf{Q}}^{\cyc}/{\bf{Q}}))$-torsion, 
where ${\bf{Q}}^{\cyc}$ denotes the cyclotomic ${\bf{Z}}_p$-extension of ${\bf{Q}}$. It then follows from
a simple restriction argument, using Artin formalism for abelian $L$-functions, that the analogous structure theorem holds 
for $E$ in the cyclotomic ${\bf{Z}}_p$-extension of any abelian number field. In particular, $X(E/k^{\cyc})$ is $\Lambda(\Gamma)$-torsion, 
and hence has a $\Lambda(\Gamma)$-characteristic power series with associated cyclotomic Iwasawa invariants 
$\mu_E(k) = \mu_{\Lambda(\Gamma)}\left( X(E/k^{\cyc}) \right)$ and $\lambda_E(k) =\lambda_{\Lambda(\Gamma)}
\left( X(E/k^{\cyc}) \right)$. Using this result, we then deduce the following structure theorem for the dual Selmer group $X(E/k_{\infty})$.

\begin{theorem} Let $E/{\bf{Q}}$ be an elliptic curve with good ordinary reduction at each prime above 
$p$ in k.

\begin{itemize}

\item[(i)] (Theorem \ref{torsion}) The dual Selmer group $X(E/ k_{\infty})$ is $\Lambda(G)$-torsion, hence 
has a $\Lambda(G)$-characteristic power series $\operatorname{char}_{\Lambda(G)}X(E/k_{\infty})$.

\item[(ii)] (Theorem \ref{2mu}) If the cyclotomic invariant $\mu_E(k)$ vanishes, then the two-variable invariant 
$\mu_{\Lambda(G)}\left( X(E/k_{\infty}) \right)$ also vanishes.

\item[(iii)] (Theorem \ref{Euler}) Let $\operatorname{char}_{\Lambda(G)} X(E/k_{\infty})(0)$ denote the image of the characteristic
power series $\operatorname{char}_{\Lambda(G)}X(E/k_{\infty})$ under the augmentation map $\Lambda(G) 
\longrightarrow {\bf{Z}}_p$. If $p \geq 5$ and the $p^{\infty}$-Selmer group $\Sel(E/k)$ is finite, then

\begin{align*}
\vert \operatorname{char}_{\Lambda(G)}X(E/k_{\infty})(0)\vert_p &=  \frac{\vert E(k)_{p^{\infty}} \vert^2 }{\vert \Sha(E/k)(p)\vert} \cdot 
\frac{\prod_v \vert c_v \vert_p}{\prod_{v \mid p}\vert \widetilde{E}_v(\kappa_v)(p)\vert^2 }.
\end{align*} Here, $\Sha(E/k)(p)$ denotes the $p$-primary part of the Tate-Shafarevich group $\Sha(E/k)$ of $E$ over $k$, 
$E(k)_{p^{\infty}}$ the $p$-primary part of the Mordell-Weil group $E(k)$, $\kappa_v$ the residue field at $v$, $\widetilde{E}_v$ the
reduction of $E$ over $\kappa_v$, and $c_v=[E(k_v):E_0(k_v)]$ the local Tamagawa factor at a prime 
$v \subset \mathcal{O}_k$. 

\item[(iv)] (Theorem \ref{mhg}) If $\mu_E(k)=0$, then there is an isomorphism of $\Lambda(H)$-modules $X(E/k_{\infty}) \cong \Lambda(H)^{\lambda_E(k)}$.

\end{itemize}

\end{theorem} We also obtain from this the following application to Tate-Shafarevich ranks.
Consider the short exact descent sequence of discrete $\Lambda(H)$-modules
\begin{align*} 0 \longrightarrow E(k_{\infty})\otimes {\bf{Q}}_p/{\bf{Z}}_p \longrightarrow \Sel(E/k_{\infty})
\longrightarrow \Sha(E/k_{\infty})(p) \longrightarrow 0. \end{align*} Here, $E(k_{\infty})$ denotes the Mordell-Weil group 
of $E$ over $k_{\infty}$, and $\Sha(E/k_{\infty})(p)$ denotes the $p$-primary part of the Tate-Shafarevich
group of $E$ over $k_{\infty}$. 

\begin{proposition}[Proposition \ref{tsrank}] Assume that $p$ is odd, and moreover that $p$ does not divide the class number 
of $k$ if the root number $\epsilon(E/k, 1)$ equals $-1$. If $E$ has good ordinary reduction at $p$ with $\mu_E(k)=0$, then
\begin{align*}\operatorname{corank}_{\Lambda(H)}\Sha(E/k_{\infty})(p) = \begin{cases}\lambda_E(k) &\text{if $\epsilon(E/k, 1)=+1$}\\
\lambda_E(k) -1&\text{if $\epsilon(E/k, 1)=-1.$}\end{cases} \end{align*} \end{proposition}

\begin{example} Consider the elliptic curve $E=53a: y^2 +xy +y = x^3 -x^2$ at $p=5$ over $k={\bf{Q}}(\sqrt{-31})$.
The discriminant of $k$ is $-31$, which is prime to both $5$ and the conductor $53$ of $E$. A simple calculation
shows that the root number $\epsilon(E/k, 1)$ is $+1$. Moreover, the mod $5$ Galois representation associated to $E$ is 
surjective, as shown by the calculations in Serre \cite[$\S$ 5.4]{Se}. Computations of Pollack \cite{Po} show that $\mu_E(k) =0$ with 
$\lambda_E(k)=9$ (and moreover that the Mordell-Weil rank of $E(k)$ is $1$), from which we deduce that $\Sha(E/k_{\infty})(5)$ 
has $\Lambda(H)$-corank $9$. In particular, $\Sha(E/k_{\infty})(5)$ contains infinitely many copies of ${\bf{Q}}_5 / {\bf{Z}}_5$.\end{example}

Finally, we establish the following criterion for one divisibility of (iii) in terms of specializations to cyclotomic characters, following a 
suggestion of Ralph Greenberg. To be more precise, let $\Psi$ denote the set of finite order characters of the Galois group 
$\Gamma = \Gal(k^{\cyc}/k)$. Given a character $\psi \in \Psi$, let us write $\mathcal{O}_{\psi}$ to denote the ring 
obtained from adjoining to ${\bf{Z}}_p$ the values of $\psi$. Let $L_p(E/k_{\infty})\vert_{\Omega}$ denote the image
of the two-variable $p$-adic $L$-function $L_p(E/k_{\infty})$ in the Iwasawa algebra $\Lambda(\Omega)$.

\begin{theorem}[Corollary \ref{gpw}] Assume that $p$ does not divide $L_p(E/k_{\infty})\vert_{\Omega}$,
and that for each character $\psi \in \Psi$, we have the inclusion of ideals \begin{align}\label{HDiv}
\left( \psi \left( L_p(E/k_{\infty}) \right) \right) \subseteq \left( \psi \left( \operatorname{char}_{\Lambda(G)}X(E/k_{\infty}) \right) \right) \text{ in } \mathcal{O}_{\psi}[[G]].
\end{align} Then, we have the inclusion of ideals \begin{align}\label{2VDiv}
\left( L_p(E/k_{\infty}) \right) \subseteq \left( \operatorname{char}_{\Lambda(G)}X(E/k_{\infty}) \right) \text{ in } \Lambda(G).
\end{align}

\end{theorem}  We deduce from this the following result. Let $K$ be any finite extension of $k$ contained in the cyclotomic ${\bf{Z}}_p$-extension $k^{\cyc}$.
Let $\Omega_K$ denote the Galois group $\Gal(KD_{\infty}/k)$, which is topologically isomorphic to ${\bf{Z}}_p$. Let $L_p(E/k_{\infty})\vert_{\Omega_K} $ denote
the image of the two-variable $p$ adic $L$-function $L_p(E/k_{\infty})$ in the Iwasawa algebra $\Lambda(\Omega_K)$. Let $\Psi_K$ denote the set of
characters of order $[K:k]$ of the Galois group $\Gal(K/k)$. Let us consider as well the following condition(s), so that we can invoke the recent work of Pollack-Weston \cite{PW}.

\begin{hypothesis}\label{XXX} Let $\epsilon(E/k, 1) \in \lbrace \pm 1 \rbrace$ denote the root number of the complex $L$-function 
$L(E/k, s) = L(f \times \theta_k, s)$. We assume that:

\begin{itemize}

\item[(i)] The mod $p$ Galois representation $\overline{\rho}_E$ associated to $E$  is surjective.

\item[(ii)] If $\epsilon(E/k, 1) = +1$, then $p \geq 5$ and the conductor $N$ is prime to the discriminant of $k$. 
This latter condition determines an integer factorization $N = N^+ N^{-}$ of $N$, where $N^+$ is divisible only by primes that 
split in $k$, and $N^{-}$ is divisible only by primes that remain inert in $k$; we then assume that $N^{-}$ is the squarefree product of an
odd number of primes. 

\end{itemize}\end{hypothesis} We obtain the following main result.
\begin{proposition}[Proposition \ref{bccrit}] 

Assume that the root number $ \epsilon(E/k, 1)$  of $L(E/k, 1)$ is $+1$.
Assume additionally that for a finite extension $K$ of $k$ contained in the cyclotomic ${\bf{Z}}_p$-extension $k^{\cyc}$, 
we have the inclusion of ideals \begin{align}\label{BCDiv} \left( L_p(E/k_{\infty})\vert_{\Omega_K} \right) \subseteq 
\left( \operatorname{char}_{\Lambda(\Omega_K)}X(E/KD_{\infty}) \right) \text{ in } \Lambda(\Omega_K),
\end{align}  with equality for $K=k$. Then, there exists a nontrivial character $\psi \in \Psi_K$ such that the specialization divisibility 
$(\ref{HDiv})$ holds. In particular, if Hypothesis \ref{XXX} (i) and (ii) hold, then we obtain the inclusion of ideals \begin{align*}
\left( L_p(E/k_{\infty}) \right) \subseteq \left( \operatorname{char}_{\Lambda(G)}X(E/k_{\infty}) \right) \text{ in } \Lambda(G).
\end{align*}\end{proposition} Though we do not discuss the issue here, the equality condition for $k=K$ would follow from the
nonvanishing criterion of Howard \cite[Theorem 3.2.3 (c)]{Ho} for dihedral/anticyclotomic $p$-adic $L$-functions, as explained 
in \cite[$\S 5$]{VO}. Hence, by Proposition \ref{bccrit}, this criterion would also imply one divisibility
of the two-variable main conjecture in the setting where the root number $\epsilon(E/k,1)$ is $1$.

\begin{remark}[Acknowledgements.] It is a pleasure to thank John Coates, Ralph Greenberg, David Loeffler, Robert Pollack, Christopher 
Skinner and Christian Wuthrich for various helpful discussions. In particular, it is a pleasure to thank Christopher Skinner for informing me
of the three-variable main conjecture proved in \cite{SU}, which I had not been aware of before writing this work. It is also a pleasure to thank 
the anonymous referee for various helpful comments that have done much to improve the exposition, as well as the correctness of some 
of the writing. \end{remark}

\tableofcontents

\section{Two-variable $p$-adic $L$-functions} 

We start with the proof of Theorem \ref{2vinterpolation}, following closely the constructions
of Hida \cite{Hi} and Perrin-Riou \cite{PR0}. Both of the these constructions depend in an 
essential way on the bounded linear form defined in \cite{Hi}, which we review below.

\begin{remark}
The results described below hold more generally for $f$ any $p$-ordinary eigenform of weight
$l \geq 2$ and nontrivial Nebentypus, following the same methods described below 
with \cite[Th\'eor\`eme B]{PR0}. We have restricted to the setting of eigenforms 
associated to modular elliptic curves for simplicity of exposition.\end{remark}  

\begin{remark}[Hida's bounded linear form.]

We follow Hida \cite[$\S 4$]{Hi}, using the same notations for spaces 
of modular forms and Hecke algebras used there. Suppose we have a modular form
\begin{align*} f(z) = \sum_{n\geq0} a_n(f) e^{2\pi i nz} \in M_l(\Gamma_{*}(M),
\xi; L_0), \end{align*} with $l$ and $M$ positive integers, $* = 0$ or $1$,
$\xi$ a Dirichlet character mod $M$, and $L_0 =
{\bf{Q}}(a_n(f))_{n\geq0}$ the extension of ${\bf{Q}}$ generated by
the Fourier coefficients of $f$. We define a norm
$\vert \cdot \vert_p$ on $f \in M_l(\Gamma_{*}(M), \xi; L_0)$ by
letting \begin{align*} \vert f \vert_p = \sup_n \left| a_n(f)
\right|_p. \end{align*} Let $L$ denote the closure of $L_0$ in
$\overline{\bf{Q}}_p$ with respect to a fixed embedding
$\overline{\bf{Q}} \rightarrow \overline{\bf{Q}}_p$. Let
$M_l(\Gamma_{*}(M), \xi; L)$ denote the completion of the space
$M_l(\Gamma_{*}(M), \xi; L_0)$ with respect to $\vert \cdot
\vert_p$. Let $\mathcal{O} = \mathcal{O}_L$. Define a subspace of
{\it{integral forms}} \begin{align*} M_l(\Gamma_{*}(M), \xi; \mathcal{O}) = \lbrace f \in
M_l(\Gamma_{*}(M), \xi; L): \vert f \vert_p \leq 1\rbrace. \end{align*} Let us write
${\bf{T}}(M, \xi; L)$ to denote the algebra of Hecke operators acting on
$M_l(\Gamma_{*}(M), \xi; L),$ as defined in \cite[p. 171]{Hi}. Hence, ${\bf{T}}(M, \xi; L)$
denotes the $L$-subalgebra of the ring of all $L$-linear endomorphisms of $M_l(\Gamma_{*}(M), \xi; L)$
generated by Hecke operators. If given integers $n \geq m \geq 0$, then the
restriction ${\bf{T}}(Mp^n, \xi; \mathcal{O})$ of ${\bf{T}}(Mp^n, \xi; L)$ to
$M_l(\Gamma_{*}(Mp^m), \xi; \mathcal{O})$ defines an
$\mathcal{O}$-algebra homomorphism
\begin{align*} M_l(\Gamma_{*}(Mp^n), \xi; \mathcal{O}) \longrightarrow
M_l(\Gamma_{*}(Mp^m), \xi; \mathcal{O}). \end{align*} We define the
{\it{extended Hecke algebra}} by passage to the inverse limit with
respect to these homomorphisms, \begin{align*} {\bf{T}}(M, \xi; \mathcal{O}) = \ilim n {\bf{T}}(Mp^n, \xi,
\mathcal{O}). \end{align*} Let us now fix a $p$-ordinary eigenform 
\begin{align*}f(z) = \sum_{n\geq1} a_n(f) e^{2\pi i nz} \in S_2(\Gamma_0(N)) \end{align*}
of weight $2$, level $N$, and trivial Nebentypus. Let $\Psi$ denote the principal or trivial character
modulo $N$ (hence $\psi(p) =1$ if $p$ does not divide $p$, and $\psi(p)=0$ otherwise). Let $\alpha_{p}(f)$ denote 
the $p$-adic unit root of the polynomial \begin{align*} x^2 - a_p(f)x + p\psi(p), \end{align*} and $\beta_p(f)$
the non-unit root. Let $f_0$ denote $p$-stabilization of $f$, which is the unique ordinary form associated to $f$ by 
Hida \cite[Lemma 3.3]{Hi}. That is, let  \begin{align*} f_0(z) = \begin{cases} f(z) &\text{if $p \mid N$}\\
f(z) - \beta_p(f) f(pz) &\text{if $p \nmid N.$}  \end{cases}\end{align*} This eigenform $f_0$ has 
level $N_0$, where \begin{align*} N_0 = \begin{cases} Np
&\text{if $p \nmid N$}\\ N &\text{if $p \mid N.$}\end{cases}\end{align*}
Its Fourier coefficients $a_n(f_0)$ satisfy the relations \begin{align*}a_n(f_0) 
= \begin{cases} a_n(f) &\text{if $(n,p)=1$}\\ \alpha_p(f) &\text{if
$n=p.$} \end{cases} \end{align*}  We now recall briefly the definition of idempotent operators 
in extended Hecke algebras, following \cite[pp. 171 - 172]{Hi}.
That is, let ${\bf{T}}(Np^m) = {\bf{T}}(\Gamma_0(Np^m);
\mathcal{O})$ denote the $\mathcal{O}$-algebra generated by Hecke operators 
acting on the space of cusp forms $S_2(\Gamma_0(Np^m);\mathcal{O})$,
with $T_p = T_p(Np^m)$ denoting the Hecke operator at $p$. 
Let $\overline{T}_{p}$ denote the image of $T_p$ in the quotient
${\bf{T}}(Np^m)/p$. This reduction $\overline{T}_p$ can be decomposed
uniquely into semisimple and nilpotent parts. Since
${\bf{T}}(Np^m)/p$ is a finitely-generated, commutative
${\bf{F}}_p$-algebra, it follows that $\overline{T}_{p}^{~p^r}$ is
semisimple for $r$ sufficiently large. Hence, $\overline{T}_{p}^{~
up^r}$ is idempotent for some integer $u$. Let $e_m$ denote the
unique lift to ${\bf{T}}(Np^m)$ of $\overline{T}_{p}^{~ up^r}$. Note that this lift
does not depend on the choice of integer $u$. 

\begin{definition}The {\it{idempotent}}
${\bf{e}}$ in the extended Hecke algebra ${\bf{T}}(N) = \ilim m
{\bf{T}}(Np^m)$ is defined to be the projective limit $ 
{\bf{e}} = \ilim m e_m. $\end{definition}

\begin{proposition}\emph{(Hida)} Let 
$f \in S_2(\Gamma_0(N))$ be a $p$-ordinary eigenform, with 
$f_0$ its associated ordinary form. There is a decomposition ${\bf{T}}(N;L) \cong
A \oplus L$ induced by the split exact
sequence \begin{align}\label{hida4.4}\begin{CD}0 @>>> A \oplus L
@>>> {\bf{T}}(N;L) @>{\phi(f_0)}>> {\bf{T}}^{(0)}(N;L) @>>>0.
\end{CD} \end{align} Here, $\phi(f_0)$ is the map
that sends $T_n \longmapsto a_n(f_0)$, with ${\bf{T}}^{(0)}(N;L) \cong L$ the
direct summand of ${\bf{T}}(N;L)$ through which this map factors, and $A$
the complementary direct summand. \end{proposition}

\begin{proof}
See \cite[Proposition 4.4 and (4.5)]{Hi}. $\Box$
\end{proof} We now use this result to define the following operator.

\begin{definition} Let  $f \in S_2(\Gamma_0(N))$ be a $p$-ordinary 
eigenform with associated ordinary form $f_0.$
We let ${\bf{1}}_{f_0}$ denote the component of the idempotent
${\bf{e}}$ corresponding to the summand ${\bf{T}}^{(0)}(N)$ in
the spit exact sequence $(\ref{hida4.4})$ above.
\end{definition}

\begin{definition} Let  $f \in S_2(\Gamma_0(N))$ be a $p$-ordinary eigenform 
with associated ordinary form $f_0.$ Let $m \geq 0$ be an integer. {\it{Hida's 
bounded linear form $l_{f_0}$}} of level $Np^m$ is then
given by the map \begin{align*}l_{f_0}: M_2(\Gamma_{*}(Np^m), \xi; L) \longrightarrow L,
~~~ g \longmapsto a_1\left( g\vert_{
{\bf{e}}\circ{\bf{1}}_{f_0}} \right),\end{align*} in other words by the map that sends
a modular form $g \in M_2(\Gamma_{*}(Np^m), \xi; L)$ to the the first
Fourier coefficient of its image under the operation ${\bf{e}}\circ
{\bf{1}}_{f_0}$. \end{definition}

\begin{proposition}\emph{(Hida)}
The linear form $l_{f_0}: M_2(\Gamma_{*}(Np^n), \xi; L)
\longrightarrow L$ is given explicitly on any  $g \in M_2(\Gamma_{*}(Np^m), \xi; L)$
by the map \begin{align*} g \longmapsto \alpha_p(f_0)^{-m} \cdot p \cdot \frac{ \langle h_m, g \rangle_{Np^m}}{ \langle
h ,f_0 \rangle_{N_0}}. \end{align*} Here, $h = \overline{f}_0(z) |_2 \left(\begin{array}
{cc} 0& -1\\ N_0 & 0 \end{array}\right)$ with $\overline{f}_0(z) = \overline{f_0(- \overline{z})},$ and
$h_m(z) = h(p^mz)$.
\end{proposition}

\begin{proof}
See \cite[Proposition 4.5]{Hi}. $\Box$
\end{proof}

\begin{lemma}\label{hidaintegrality}
The linear form $l_{f_0}$ sends $M_2(\Gamma_{*}(Np^m), \xi;
\mathcal{O})$ to $ \mathcal{O}.$
\end{lemma}

\begin{proof}
Fix $g \in M_2(\Gamma_{*}(Np^m), \xi; \mathcal{O})$. We know 
that $\vert \alpha_p(f) \vert_p = \vert
a_p(f_0)\vert_p = 1.$ On the other hand, the operator $\phi(f_0)$ in the 
split exact sequence $(\ref{hida4.4})$ sends
$T_p(Np^m) \longmapsto a_p(f_0)$ for each $m \geq 0$. It follows
that $\phi(f_0)$ sends the idempotent ${\bf{e}} = \ilim m e_m$ to the
unit defined by $\lim_r a_p(f_0)^{p^r}= \lim_r \alpha_p(f_0)^{p^r}.$
Now, the action of ${\bf{T}}(N)$ maps the space
$M_2(\Gamma_{*}(Np^m); \mathcal{O})$ to itself for any $m\geq 0,$ as
explained for instance in \cite[$\S 4$]{Hi}. Thus if $\vert g
\vert_p \leq 1$, then $g \vert_{ {\bf{e}}\circ {\bf{1}}_{f_0}} =
\left( g\vert_{{\bf{e}}}\right)\vert_{{\bf{1}}_{f_0}}$ has the
property that $\left| a_1\left( g \vert_{ {\bf{e}}\circ
{\bf{1}}_{f_0}} \right) \right|_p \leq 1$. The result follows. $\Box$
\end{proof} \end{remark}

\begin{remark}[Some $p$-adic convolution measures.]

We now give a sketch of Perrin-Riou's construction of the measure $d\mu_f$, \cite{PR0}, 
starting with the setup described above. This construction is made up of several
constituent measures that a priori take values in the spaces $M_l(\Gamma_{*}(M),
\xi; L)$, but can be seen to take values in the integral subspaces
$M_l(\Gamma_{*}(M), \xi; \mathcal{O}),$ as we show in Proposition \ref{integrality}. 

Let us fix throughout a finite order character $\mathcal{W}$ of $G$. We commit an abuse of
notation in viewing $\mathcal{W}$ as a character on the ideals of $k$ via class field theory. Observe that we can
always write such a character $\mathcal{W}$ as the product of characters $\rho \chi \circ {\bf{N}}$, where $\rho$ is a character of $G$ 
that factors though the dihedral ${\bf{Z}}_p$-extension $D_{\infty}$ of $k$, and $\chi \circ {\bf{N}}$ a character of $G$ that factors 
though the cyclotomic ${\bf{Z}}_p$-extension $k^{\cyc}$ of $k$. Here, the cyclotomic character $\chi \circ {\bf{N}}$ is given by the 
composition with the norm homomorphism ${\bf{N}}$ on ideals of $k$ with some Dirichlet character $\chi$ that factors through the
cyclotomic ${\bf{Z}}_p$-extension of ${\bf{Q}}$. Hence, we fix a finite order character $\mathcal{W}$ of $G$ with 
dihedral/cyclotomic factorization \begin{align} \label{decomposition}\mathcal{W}  = \rho \chi \circ {\bf{N}}. \end{align} 
Let $c(\mathcal{W})$ denote the conductor of $\mathcal{W}$, with $c(\rho)$ the conductor of the dihedral or anticyclotomic part $\rho$. 
Let $D = D_{k/{\bf{Q}}}$ denote the discriminant of $k$. Let $\omega = \omega_{k/{\bf{Q}}}$ denote the quadratic character
associated to $k$. A classical construction associates to the $\mathcal{W}$ a theta series of weight $1$, level $\Delta = \Delta(\mathcal{W})= 
\vert D\vert {\bf{N}}c(\mathcal{W})^2,$ and Nebentypus $\omega \chi^2.$ To be more precise, let $\mathcal{O}_{c(\rho)} = {\bf{Z}} + c(\rho)\mathcal{O}_k$ 
denote the ${\bf{Z}}$-order of conductor $c(\rho)$ in $\mathcal{O}_k$. Fix an element of the class group $ A \in\Pic\mathcal{O}_{c(\rho)}$. Fix a representative
$\mathfrak{a} \in A$. We then define a $\chi$-twisted theta series associated to $A$, \begin{align*}\Theta_A(\chi)(z) &= \frac{1}{u} \sum_{x \in \mathfrak{a}} 
\chi \left( \frac{{N_{k/ {\bf{Q}}}}(x) }{ {\bf{N}}\mathfrak{a} } \right) e^{\frac{2\pi iN_{k/ {\bf{Q}} }(x) z}{ {\bf{N}}\mathfrak{a} }} = \frac{1}{u} +
\sum_{n\geq 1} \chi(n) r_A(n) e^{2\pi i nz}. \end{align*} Here, $x$ runs over points in the lattice defined by $\mathfrak{a}$, $ u = 2 \vert \mathcal{O}_{k}^{\times}\vert$ 
is twice the number of units of $k$ , and $r_A(n)$ is the number of ideals of norm $n$ in $A$. This series does not depend on choice of representative $\mathfrak{a} \in A$. 
It is seen to lie in $M_1(\Gamma_0(\Delta), \omega\chi^2)$ by a standard application of Poisson summation. Taking the $\rho$-twisted sum over classes 
$A \in \Pic\mathcal{O}_{c(\rho)},$ it gives rise to a modular form \begin{align*}\Theta(\mathcal{W})(z) = \sum_{A}\rho(A) \Theta_A(\chi)(z) \in
M_1(\Gamma_0(\Delta), \omega\chi^2)\end{align*} associated to
$\mathcal{W}$. We refer the reader to $\cite{GZ}$, \cite{Hi} or $\cite{He}$ for proofs of
these facts. In what follows, we fix a finite order character $\mathcal{W}$
of $G$ having the decomposition $(\ref{decomposition})$ above. We fix a ring class $A \in  \Pic\mathcal{O}_{c(\rho)}$. 
We then construct a measure associated to the underlying Dirichlet character $\chi$ in the decomposition $(\ref{decomposition}).$ 
In fact, to follow \cite{PR0}, we shall suppose more generally that $\chi$ is any finite order character of ${\bf{Z}}_p^{\times}.$
Taking the $\rho$-twisted sum over classes $A \in \Pic\mathcal{O}_{c(\rho)}$ then gives the 
appropriate measure in $\mathcal{O}[[G]]$ whose specialization to $\mathcal{W}$ interpolates the value \begin{align*} 
\frac{L(f \times \Theta(\overline{\mathcal{W}}), 1)}{8 \pi^2 \langle f ,f \rangle_N} \in \overline{\bf{Q}}_p 
\end{align*} up to some algebraic factor (which can be made explicit).  We give only a sketch of this construction, 
referring the reader to \cite{PR0} for proofs and calculations. We start with the following constituent constructions.

\begin{definition}[Theta series measures.] 

Fix an integer $m \geq 1$. Consider the series defined by \begin{align*}
\Theta_A(\chi)(a, p^m)(z) = \sum_{x \in \mathfrak{a} \atop \frac{ N_{k/{\bf{Q}}}(x)}{ {\bf{N}}(\mathfrak{a})}
\equiv a \operatorname{mod} p^m} \chi \left( \frac{ N_{k/{\bf{Q}}}(x)
}{ {\bf{N}}\mathfrak{a} } \right) e^{\frac{2\pi i N_{k/{\bf{Q}}}(x) z
}{ {\bf{N}}\mathfrak{a} } } .\end{align*} Let $d\Theta_A(\chi)$ denote the
measure on ${\bf{Z}}_{p}^{\times}$ given by the rule
\begin{align*}\int_{a + p^m{\bf{Z}}_p^{\times}} d\Theta_A(\chi) = \Theta_A(\chi)(z).\end{align*}

\end{definition} \begin{lemma}\label{thetaintegrality}
The measure $d\Theta_A(\chi)$ takes values in the space $M_1(\Gamma_0(\Delta),
\omega\chi^2; \mathcal{O})$ if $p \geq 5$.
\end{lemma} \begin{proof} The result
follows plainly from the $q$-expansion of $\Theta_A(\chi)(z)$.$\Box$ \end{proof} 

\begin{remark}  We impose the condition 
$p \geq 5$ to deal with the $\frac{1}{u}$ term in the $q$-expansion of $\Theta_A(\chi)(z)$, since we could 
in exceptional cases have $u=4$ or $u=6$. \end{remark}

\begin{definition} [Eisenstein series measures.] 

Let $\xi$ be an odd Dirichlet character modulo an integer $M >2.$ Let $E_M(\xi)$ denote the Eisenstein 
series of weight $1$ given by \begin{align*} E_M(\xi)(z) = \frac{L(\xi, 0)}{2} + \sum_{n \geq 1}\left( \sum_{d
>0 \atop d \mid n} \xi(d)\right) e^{2\pi i n z}. \end{align*} Here, \begin{align*}L(\xi, s)
= \sum_{n\geq 1} \xi(n) n^{-s}\end{align*} is the standard Dirichlet $L$-series associated to $\xi$.
The series $E_{M}(\xi)(z)$ lies in $M_1(\Gamma_0(M), \xi)$, as shown for instance in \cite{H2}.
Fix an integer $m \geq 1.$ Let $M = Np^m$. Consider the series defined by \begin{align*}E(\xi)(a, M)(z) =
\frac{L(\xi, 0)}{2} + \sum_{n \geq 1}\left( \sum_{d >0 , d \mid n
\atop d \equiv a \operatorname{mod} M} \xi(d)\right) e^{2\pi i n
z}.\end{align*} Fix an integer $C>1$ prime to $M$. Let $C^{-1}$ denote the 
inverse class of $C$ modulo $M$. Consider the difference defined by \begin{align*}E^{C}(\xi)(a, M)(z) = 
E(\xi)(a, M)(z) - C E(\xi)(C^{-1}a, M)(z).\end{align*} It is well known that $E^{C}(\xi)(a, M)(z)$ is a
bounded distribution on the product ${\bf{Z}}_{p}^{\times} \times \left({\bf{Z}}/N\right)^{\times}$ (see \cite{Hi}, 
\cite{Ka} or \cite{Ka2}). Let $dE^{C}(\xi)(a, M)$ denote the measure on
${\bf{Z}}_{p}^{\times} \times \left( {\bf{Z}}/N\right)^{\times}$
given by the rule \begin{align*} \int_{a + Np^m{\bf{Z}}_p^{\times}} dE^C(\xi) (a, M)
= E^C(\xi)(a, Np^m)(z).\end{align*} Note that this measure takes values in certain
spaces of Eisenstein series. To be more precise, we have the following result.

\begin{lemma}\label{eisensteinintegrality}
The measure $dE^{C}(\xi)(a, M)$ takes values in the space $M_1(\Gamma_0(M),\xi; \mathcal{O})$.
\end{lemma}

\begin{proof}The result follows from the Key Lemma of Katz \cite[1.2.1,
Key Lemma for $\Gamma(N)$]{Ka}, which shows that the Eisenstein measure
takes $p$-integral values at an elliptic curve with level structure
defined over a $p$-integral ring. Note also that
$dE^{C}(\xi)(a, M)$ arises from a one-dimensional part of the Eisenstein
pseudo-distribution $2H^{(a,b)}$ given in \cite[$\S 3.4$]{Ka} (i.e.
with $a=C$). This pseudo-distribution can be shown to take integral
values by \cite[Key Lemma 1.2.1]{Ka}, e.g. by the proof given in
\cite[Theorem 3.3.3]{Ka} (cf. also \cite[$\S 3.5,(3.5.5)$]{Ka}).
$\Box$ \end{proof} \end{definition}

\begin{definition}[Convolution measures.] 

Fix a class $A \in \Pic \mathcal{O}_{c(\rho)}.$
Fix integers $a, m \geq 1.$ Fix an integer $C>1$ prime to
$pND.$ Consider the series defined by
\begin{align*} \Phi^C_A(\chi)(a, p^m)(z) = \sum_{\alpha \in
\left({\bf{Z}}/N\Delta\right)^{\times}} \Theta_A(\chi)(\alpha^2 a,
p^m)(Nz) E^C(\omega\chi^2)(\alpha, N\Delta)(z).\end{align*} The function
$\Phi_{A}^{C}(a, p^m)(z)$ can be seen to define a bounded
distribution on ${\bf{Z}}_{p}^{\times}$ (see \cite[Lemme 4]{PR0}).
Let $d\Phi_{A}^{C}(\chi)$ denote the measure on
${\bf{Z}}_{p}^{\times}$ given by this function.

\begin{lemma}\label{convolutionintegrality} The measure
$d\Phi_{A}^{C}(\chi) = \Phi_{A}^{C}(a, p^m)(z)$ takes values in the space
$M_2(\Gamma_0(N\Delta), \omega\chi^2; \mathcal{O})$, at least if $p \geq 5$. 
\end{lemma}

\begin{proof}
The function $\Phi_A(a, p^m)(z)$ lies in $M_2(\Gamma_0(N\Delta),
\omega\chi^2)$ (see \cite[Lemme 5]{PR0}). We then deduce from Lemmas
\ref{thetaintegrality} and \ref{eisensteinintegrality} that it lies
in $M_2(\Gamma_0(N\Delta), \omega\chi^2; \mathcal{O})$. $\Box$
\end{proof} \end{definition}

\begin{definition}[Trace operators.] 

Keep the setup used to define the convolution measure $d\Phi_A^C(\chi)$ above. Fix a set representatives $\mathcal{R}$ 
for the space $\Gamma_0(N\Delta)\backslash \Gamma_0(N).$ We define the trace operator $\Tr_{N}^{N\Delta}: 
M_2(\Gamma_0(N\Delta), \xi)\longrightarrow M_2(\Gamma_0(N), \xi)$ by the rule
\begin{align*} h(z) \longmapsto \sum_{\gamma \in \mathcal{R}} \xi(a_{\gamma}) \cdot h\vert_2
\gamma, ~~~ \gamma = \left(\begin{array} {cc} a_{\gamma}& b_{\gamma} \\
c_{\gamma} & d_{\gamma} \end{array}\right). \end{align*} 

\begin{lemma}\label{traceintegrality} The composition function
$ \Tr_{N}^{N\Delta} \circ \Phi_{A}^{C}(\chi)(a,
p^m)(z)$ takes values in the space $M_2(\Gamma_0(N),
\omega\chi^2; \mathcal{O})$, at least if $p \geq 5$. \end{lemma}

\begin{proof}
Given the result of Lemma \ref{traceintegrality}, the assertion can be deduced from explicit 
computations of the Fourier series expansion of the trace form. If $N$ and $D$ are
both prime, then the result follows from the computation given in
Gross \cite[Proposition 9.3, 2)]{G}. In the more general case with 
$(N, D)=1$, it follows from the computation of the coefficients given in
Gross-Zagier \cite[IV$\S2$ Proposition (2.4) and $\S3,$ Proposition
(3.2)]{GZ}. $\Box$
\end{proof} \end{definition}

\begin{definition}[Mesures fondamentales.] 

Keep the setup from above. Recall that we let $f_0$ denote the $p$-stabilization
of $f$, which is the unique ordinary form associated to $f$ by Hida \cite[Lemma 3.3]{Hi}.
Let $l_{f_0}: M_2(\Gamma_0(N), \omega\chi^2; L) \longrightarrow L$
denote Hida's bounded linear form, as defined above. Let $L$ denote the closure of the field of values 
$L_0= {\bf{Q}}(\omega\chi^2(n))_{n\geq0}$ in $\overline{\bf{Q}}_p$. Let $d\phi_{A}^{C}(\chi)$ denote the 
measure on ${\bf{Z}}_{p}^{\times}$ given by the rule \begin{align*} \int_{a + p^m{\bf{Z}}_p^{\times}}
d\phi_{A}^{C}(\chi) = l_{f_0} \circ \Tr_{N}^{N\Delta} \circ
\Phi_{A}^{C}(\chi)(a, p^m)(z). \end{align*}

\begin{proposition} \label{integrality} The measure $d\phi_{A}^{C}(\chi)$ takes values
in the ring $\mathcal{O} = \mathcal{O}_L$, at least if $p \geq 5$.
\end{proposition}

\begin{proof}
The result follows from Lemmas \ref{thetaintegrality},
\ref{eisensteinintegrality}, \ref{convolutionintegrality},
\ref{traceintegrality} and \ref{hidaintegrality}. $\Box$
\end{proof} \end{definition} We can now at last define the two-variable measures
that gives rise to $d\mu_f$.

\begin{definition} Keep the notations above, with $C>1$ an integer
prime to $pND$. Let $L(\rho)$ denote the closure of the field
of values $L_0(\rho(A))_{A \in \Pic \mathcal{O}_{c(\rho)}}$ in $\overline{\bf{Q}}_p$.
Let $\mathcal{O} = \mathcal{O}_{L(\rho)}$. Let $d\mu_f^C$ denote the $\mathcal{O}$-valued
function on $G$ defined by the rule \begin{align*}
\int_G \mathcal{W} \mu_f^C = \sum_{A \in \Pic\mathcal{O}_{c(\rho)} } \rho(A)
d\phi_A^C(\chi).\end{align*} This function is seen easily to be a well-defined distribution
on $G$ (see \cite[$\S$ 5]{PR0}), and hence  (by Proposition \ref{integrality}) a measure on $G$.
That is, the distribution is seen easily to be bounded for any choice of $p$, and integral for any choice
$p \geq 5$. It is also seen to be integral for any choice of $p$ if $\rho \neq {\bf{1}}$ (in which case the twisted sum of theta series 
$\sum_A \rho(A)\Theta_A(\chi)(z)$ is cuspidal).

\end{definition} \end{remark}

\begin{remark}[Interpolation properties and functional equation.]

Let us keep all of the notations above, with $C>1$ an integer
prime to $pND$. The two-variable measure $d\mu_f^C$ 
satisfies the following interpolation property. Let $\tau(\mathcal{W})$ denote
the Artin root number of $L(\mathcal{W}, s)$. Recall that $\Delta = \Delta(\mathcal{W})$ 
denotes the level of the theta series
$\Theta(\mathcal{W})(z)$. Let $\psi$ denote the principal character modulo $N$ as above (hence, $\psi(p) =1$ if
$p$ does not divide $N$ and zero otherwise). Recall as well that we let $\alpha_p$ denote the unique $p$-adic unit root
of the polynomial $X^2-a_p(f)X - p\psi(p)$. Given an integer $r \geq 1$, let us write $\alpha_{p^r}$ to 
denote $\alpha_{p}^r$. Let us also write $N'$ to denote the prime-to-$p$ part of $N$. Let $\beta$ denote the $p$-primary 
component of the level $\Delta$ of $\Theta(\mathcal{W})$. Finally, let us commit an 
abuse of notation in using the same notations used to denote the measures defined on ${\bf{Z}}_p^{\times}$ above to denote 
the induced measures defined on ${\bf{Z}}_p$.

\begin{theorem}\label{2vinterpolation}
There exists for each integer $C>1$ prime to $pND$ an $\mathcal{O}$-valued measure $d\mu_{f}^{C}$ on $G$ 
such that for any finite order character $\mathcal{W}$ of $G$,
\begin{align*} \int_G \mathcal{W} d\mu_{f}^{C} &= \left( 1 - \frac{\psi(p)}{\alpha_{p^2}}\right)^{-1}
\left( 1 + \frac{p\psi(p)}{\alpha_{p^2}}\right)^{-1} 
 \prod_{\mathfrak{p} \mid p }
\left( 1 - \frac{\mathcal{W}(\mathfrak{p})}{\alpha_{{\bf{N}}\mathfrak{p}}} \right) 
\left( 1 -  \frac{\overline{\mathcal{W}}\psi({\bf{N}}\mathfrak{p})(\mathfrak{p})}{\alpha_{{\bf{N}}\mathfrak{p}}} \right) \\
&\times \omega(-N') \mathcal{W}(N') \left( 1 - C
\omega(C) \overline{\mathcal{W}}(C)\right)
\frac{{\Delta}^{\frac{1}{2}}}{\alpha_{p^{\beta}}}\tau(\mathcal{W})\\
& \times \frac{L(f \times \Theta(\overline{\mathcal{W}}), 1)}{8 \pi^2 \langle
f ,f \rangle_{N}}.
\end{align*} Here, the product runs over all primes $\mathfrak{p}$ of $\mathcal{O}_k$ that divide $p$. \end{theorem}

\begin{proof} See Perrin-Riou \cite[Th\'eor\`eme A]{PR0}, along with Proposition 
\ref{integrality} above. That is, fix a finite order character $\mathcal{W}$ of $G$ 
having the decomposition $(\ref{decomposition})$. Fix an integer $C >1$
prime to $pND$. A simple argument shows that $d\mu_f^C$ is a well-defined
distribution on $G$ (see \cite[$\S$ 5]{PR0}). On the other hand, we know that 
$d\mu_f^C$ takes values in $\mathcal{O} = \mathcal{O}_{L(\rho)}$ (by Proposition 
\ref{integrality}). Hence, $d\mu_f^C$ is an $\mathcal{O}$-valued measure on $G$, corresponding 
to an element of the completed group ring $\mathcal{O}[[G]]$. The calculation of the interpolation
value is given in \cite[$\S$ 4]{PR0}. $\Box$ \end{proof} We may now define the two-variable $p$-adic $L$-function associated
to a $p$-ordinary eigenform $f \in S_2(\Gamma_0(N))$ in the tower $k_{\infty}/k,$ following Perrin-Riou \cite{PR0}. Observe
that this definition does not depend on the choice of auxiliary integer $C >1$ prime to $pND$ thanks to Theorem \ref{2vinterpolation}.

\begin{definition} Let $\eta:G \rightarrow {\bf{Z}}_{p}^{\times}$ be a continuous
character. Let $\mathfrak{D}$ denote the different of $k$. Let $C >1$
be any integer prime to $pND$. The
{\it{two-variable $p$-adic $L$-function}} $L_p(f, k)(\eta)$ of $f$ in $k_{\infty}/k$ is then
defined to be \begin{align*} L_p(f, k)(\eta) &= \left( 1 - \frac{\psi(p)}{\alpha_{p^2}}\right)
\left( 1 + \frac{p\psi(p)}{\alpha_{p^2}}\right) \\
&\times \eta^{-1}(\mathfrak{D}'N') \left( 1 -
C \omega(C) \eta^{-1}(C) \right)^{-1}\\
&\times \int_G \eta(g) d\mu_{f}^C(g).
\end{align*} Here, $\mathfrak{D}'$ and $N'$ denote the prime to $p$ parts of 
$\mathfrak{D}$ and $N$ respectively. \end{definition}

\begin{corollary} \label{FE}The function $L_p(f, k)$ is an Iwasawa function on $G$ with 
coefficients in ${\bf{Z}}_p$. Moreover, the Iwasawa function defined by 
\begin{align*}
\Lambda_p(f, k)(\eta) = \eta^{\frac{1}{2}}(N')\eta(\mathfrak{D}')L_p(f,k)(\eta)
\end{align*} satisfies the functional equation \begin{align*}
\Lambda_p(f, k)(\eta^{-1}) = - \omega(N') \Lambda_p(f,k)(\eta).\end{align*}
\end{corollary}

\begin{proof}
See \cite[Corollaire, Th\'eor\`eme A]{PR0} or \cite[Corollaire, Th\'eor\`eme B]{PR0}. $\Box$

\end{proof} \end{remark}

\section{Iwasawa module structure theory}

We now describe the Iwasawa module structure theory of the dual Selmer group of $E$
over $k_{\infty}$, along with that of the $p$-primary component of the associated Tate-Shafarevich 
group. We follow closely many of the arguments of Coates-Sujatha-Schneider \cite{CSS}, as well
as the refinements of those given by Hachimori-Venjakob \cite{HV} for the somewhat analogous 
setting of the false Tate curve extension.

\begin{remark}[Some definitions.] 

Fix $S$ a finite set of primes of $k$ containing both the primes above 
$p$ and the primes where $E$ has bad reduction. Let $k^S$ denote the maximal Galois extension of 
$k$ that is unramified outside of $S$ and the archimedean primes of $k$. Note that since $k_{\infty}$ is unramified 
outside of primes above $p$, we have the inclusion $k_{\infty} \subset k^S$. Given $L$ any finite 
extension of $k$ contained in $k_{\infty}$, let $G_S(L)$ denote the Galois group $\Gal(k^S/L)$. 
The $p^{\infty}$-Selmer group $\Sel(E/L)$ of $E$ over $L$ is defined classically as the kernel of the localization 
map, \begin{align*} \Sel(E/L) &= \ker \left( \lambda_E(L): H^1(G_S(L), E_{p^{\infty}}) \longrightarrow \bigoplus_{v \in S}J_v(L)\right).
 \end{align*} Here, $E_{p^{\infty}} = E(k^S)_{p^{\infty}}$ denotes the $p$-power torsion: $E_{p^{\infty}} = \bigcup_{n \geq 0} 
 E_{p^n}$ where $E_{p^n} = \ker ([p^n]:E \rightarrow E)$. We also write \begin{align*} J_v(L) &= \bigoplus_{w \mid v} 
 H^1(L_w, E(\overline{L}_w))(p), \end{align*} where the sum runs over all primes $w$ above $v$ in $L$. Note that
 this group fits into the classical short exact descent sequence \begin{align*}
 0 \longrightarrow E(L) \otimes {\bf{Q}}_p/{\bf{Z}}_p \longrightarrow \Sel(E/L) \longrightarrow \Sha(E/L)(p) \longrightarrow 0,
 \end{align*} where $\Sha(E/L)(p)$ denotes the $p$-primary component of the Tata-Shafarevich group $\Sha(E/L)$
 of $E$ over $L$. Let $L_{\infty}$ be any infinite extension of $k$ contained in $k_{\infty}$. We then define the Selmer 
 group of $E$ over $L_{\infty}$ to be the inductive limit \begin{align*}\Sel(E/L_{\infty}) &= \varinjlim_L \Sel(E/L). \end{align*} 
Here, the limit is taken over all finite extensions $L$ of $k$ contained in $L_{\infty}$ with respect to the natural restriction maps 
on cohomology. We write \begin{align*} X(E/L_{\infty}) &= \Hom(\Sel(E/L_{\infty}), {\bf{Q}}_p/{\bf{Z}}_p) \end{align*} to denote
the Pontryagin dual of $\Sel(E/L_{\infty})$. \end{remark}

\begin{remark}[$\Lambda(\Gamma)$-module structure.] 

Let us first review the cyclotomic structure theory implied by work of Kato and Rohrlich. 

\begin{theorem}\emph{(Kato-Rohrlich)}\label{kato-rohrlich}
If $E/{\bf{Q}}$ has good ordinary reduction at each prime above $p$ in $k$, then the dual Selmer group
$X(E/k^{\text{cyc}})$ is $\Lambda(\Gamma)$-torsion.
\end{theorem}

\begin{proof}
The result follows from the Euler system method of Kato \cite[Theorems
14.2 and 17.4]{KK}, which requires for nontriviality the
nonvanishing theorem of Rohrlich \cite{Ro}. $\Box$
\end{proof} We may then invoke the structure theory of finitely generated torsion 
$\Lambda(\Gamma)$-modules (\cite[Chapter VII, $\S 4.5$]{BB}) to obtain a 
$\Lambda(\Gamma)$-module pseudoisomorphism
\begin{align}\label{1vs} X(E/k^{\cyc})
\longrightarrow   \bigoplus_{i=1}^r \Lambda(\Gamma)/p^{m_i} \oplus
\bigoplus_{j=1 }^s\Lambda(\Gamma)/\gamma_{j}^{n_j}.\end{align} Here, the indices
$r$, $s$, $m_i$ and $n_j$ are all positive integers, and each $\gamma_j$ can be viewed
as an irreducible monic distinguished polynomial $\gamma_j(T)$ (with respect to a fixed isomorphism
$\Lambda(\Gamma) \cong {\bf{Z}}_p[[T]]$). The $\Lambda(\Gamma)$-characteristic power series
\begin{align*} \operatorname{char}_{\Lambda(\Gamma)} X(E/k^{\cyc}) &= \prod_{i=1}^r p^{m_i} 
\cdot \prod_{j=1}^s \gamma_{j}^{n_j} \end{align*} is defined uniquely up to unit in $\Lambda(\Gamma)$. 
One defines from it the $\Lambda(\Gamma)$-module invariants \begin{align*}
\mu_{\Lambda(\Gamma)}\left( X(E/k^{\cyc})\right) = \sum_{i=1}^r m_i ~~~\text{and}~~~ \lambda_{\Lambda(\Gamma)}\left(
X(E/k^{\cyc})\right) = \sum_{j=1} n_j \cdot \deg(\gamma_j).\end{align*} We shall often for simplicity denote these by 
\begin{align*}\mu_E(k) = \mu_{\Lambda(\Gamma)}\left(X(E/k^{\cyc})\right) ~~~\text{and}~~~
\lambda_E(k) = \lambda_{\Lambda(\Gamma)}\left( X(E/k^{\cyc})
\right). \end{align*} respectively.  We refer the reader to the monograph of Coates-Sujatha \cite{CS2} for further
discussion, for instance on how to compute the (finite) $G$-Euler characteristic of $\Sel(E/k^{\cyc})$, or equivalently
how to compute $\vert \operatorname{char}_{\Lambda(\Gamma)}X(E/k^{\cyc})(0) \vert_p^{-1}$, where 
$\operatorname{char}_{\Lambda(\Gamma)}X(E/k^{\cyc})(0)$ denotes the image of the characteristic power 
series $\operatorname{char}_{\Lambda(\Gamma)}X(E/k^{\cyc})$ under the natural augmentation map 
$\Lambda(\Gamma) \longrightarrow {\bf{Z}}_p$.

\end{remark}

\begin{remark}[$\Lambda(G)$-module structure.]

We now use the $\Lambda(\Gamma)$-module structure of $X(E/k^{\cyc})$
to study the $\Lambda(G)$-module structure of $X(E/k_{\infty})$, following
the main ideas of \cite{CSS} and \cite{HV}. Let us first consider the following 
standard result. Let $\mathfrak{S}(E/L)$ denote the compactified Selmer group 
of $E$ over any finite extension $L$ of $k$ contained in $k_{\infty}$, which is defined
as the projective limit \begin{align*}\mathfrak{S}(E/L) &=\varprojlim_n 
\ker \left(H^1(G_S(L), E_{p^n}) \longrightarrow \bigoplus J_v(L) \right) \end{align*} 
taken with respect to the natural maps $E_{p^{n +1}} \rightarrow E_{p^n}$ induced by multiplication
by $p$. Given any infinite extension $L_{\infty}$ of $k$ contained in $k_{\infty}$, we then define
\begin{align*} \mathfrak{S}(E/L_{\infty}) &= \varprojlim_L \mathfrak{S}(E/L)\end{align*}
to be the projective limit over all finite extensions $L$ of $k$ contained in $L_{\infty}$,
taken with respect to the natural corestriction maps.

\begin{proposition}\label{injection} Let $\Omega = \text{Gal}(L_{\infty}/k)$
be any infinite pro-$p$ group. If $E(L_{\infty})_{p^{\infty}}$ is finite,
then there is a $\Lambda(\Omega)$-module injection
\begin{align*} \mathfrak{S}(E/L_{\infty}) \longrightarrow
\text{Hom}_{\Lambda(\Omega)}(X(E/L_{\infty}), \Lambda(\Omega)). \end{align*}
\end{proposition}

\begin{proof}
See for instance \cite[Theorem 7.1]{HV}. $\Box$ \end{proof} 
We use this to deduce the following result.

\begin{theorem}\label{cycloc}
If $E$ has good ordinary reduction at each prime above $p$ in $k$, then the cohomology group
$H^2(G_S(k^{\text{cyc}}), E_{p^{\infty}})$ vanishes. In particular, the
localization map \begin{align*} \lambda_S(k^{\text{cyc}}): H^1(G_S(k^{\text{cyc}}),
E_{p^{\infty}})\longrightarrow \bigoplus_{v \in S} J_v(k^{\text{cyc}})\end{align*} 
is surjective, and hence we have a short exact sequence of $\Lambda(\Gamma)$-modules
\begin{align}\label{locsurj} 0 \longrightarrow \Sel(E/k^{\text{cyc}})
\longrightarrow H^1(G_S(k^{\text{cyc}}),E_{p^{\infty}})
\longrightarrow \bigoplus_{v \in S} J_v(k^{\text{cyc}})
\longrightarrow 0.\end{align}\end{theorem}

\begin{proof} Consider the Cassels-Poitou-Tate exact sequence
\begin{align*}0 \longrightarrow \Sel(E/k^{\text{cyc}})
&\longrightarrow H^1(G_S(k^{\text{cyc}}),
E_{p^{\infty}}) \longrightarrow \bigoplus_{v \in S}J_v(k^{\text{cyc}}) \\
&\longrightarrow \mathfrak{S}(E/k^{\text{cyc}})^{\vee}\longrightarrow H^2(G_S(k^{\text{cyc}}), E_{p^{\infty}})
\longrightarrow 0.\end{align*} Here, $\mathfrak{S}(E/k^{\cyc})^{\vee}$ is the Pontryagin dual of $\mathfrak{S}(E/k^{\cyc})$. Now, the $p$-power
torsion subgroup $E(k^{\cyc})_{p^{\infty}}$ is finite by Imai's theorem \cite{Im}. Hence, we can invoke
Proposition \ref{injection} to obtain an injection $\mathfrak{S}(E/k^{\cyc}) \rightarrow 
\Hom_{\Lambda(\Gamma)}(X(E/k^{\cyc}), \Lambda(\Gamma))$. Now, by the main result of Kato
\cite{KK}, the dual Selmer group $X(E/k^{\cyc})$ is $\Lambda(\Gamma)$-torsion. Hence, we have
a $\Lambda(\Gamma)$-module injection \begin{align*}\mathfrak{S}(E/k^{\cyc}) &\hookrightarrow \Hom_{\Lambda(\Gamma)}(X(E/k^{\cyc}),
\Lambda(\Gamma)) =0. \end{align*} It follows that $\mathfrak{S}(E/k^{\cyc})^{\vee}=0$, and hence that
$H^2(G_S(k^{\cyc}), E_{p^{\infty}})=0$. See also the argument of Kato \cite[$\S\S 13, 14$]{KK} for this
latter vanishing. $\Box$\end{proof} Let us now consider invariants
under the Galois group $H = \Gal(k_{\infty}/k^{\cyc}).$ Note that by Serre's refinement \cite{Se3} of Lazard's 
theorem \cite{Laz}, a $p$-adic Lie group with no element of order $p$ has $p$-cohomological dimension $\cd_p$
equal to its dimension as a $p$-adic Lie group. Since $G$ has no element of order $p$, we can and will invoke
this characterization throughout. Hence (for instance), $\cd_p(G)=2$ with $\cd_p(H) = \cd_p(\Gamma) =1$. To 
show the main result of this paragraph, we first establish the following standard lemmas.

\begin{lemma}\label{Hloc}
If $E$ has good ordinary reduction at each prime above $p$ in $k$, then there is a short
exact sequence \begin{align*}\begin{CD} 0 @>>>\Sel(E/k_{\infty})^H @>>> H^1(G_S(k_{\infty}), 
E_{p^{\infty}})^H \\@. @>{\eta_S(k_{\infty})}>> \bigoplus_{v \in S}J_v(k_{\infty})^H @>>> 0.\end{CD}\end{align*} 
Here, $\eta_S(k_{\infty})$ is the map induced by localization map \begin{align*}\lambda_S(k_{\infty}): H^1(G_S(k_{\infty}), 
E_{p^{\infty}}) &\longrightarrow \bigoplus_S J_v(k_{\infty}).\end{align*}\end{lemma}

\begin{proof} See \cite[Lemma 2.3]{CSS}. That is,
consider the fundamental diagram $$\begin{CD}
 0 @>>> \Sel(E/k_{\infty})^{H} @>>> H^1(G_S(k_{\infty}),
E_{p^{\infty}})^{H} @>{\eta_S(k_{\infty})}>>
\bigoplus_{v \in S}J_v(k_{\infty})^{H} \\
@.        @AAA              @AAA
@AA{\gamma_S(k^{\text{cyc}})}A     @. \\
0 @>>> \Sel(E/k^{\text{cyc}}) @>>> H^1(G_S(k^{\text{cyc}}),
E_{p^{\infty}})
@>{\lambda_S(k^{\text{cyc}})}>> \bigoplus_{v \in S}J_v(k^{\text{cyc}}).\\
\end{CD}$$ Here, the horizontal rows are exact, and the vertical arrows 
are induced by restriction on cohomology. We have that $$\coker(\gamma_S(k^{\text{cyc}})) =
\bigoplus_{w |v\in S} \coker(\gamma_w(k^{\text{cyc}})),$$ with $w$
ranging over places in $k^{\cyc}$ above $ v \in S$. Note that only
finitely many such primes exist, as no finite prime splits
completely in $k^{\cyc}$ (see for instance \cite[Theorem 2.13]{Wash}). 
Given a prime $w$  above $ v$ in $k_{\infty}$, let $\Omega_w$ denote the
decomposition subgroup of $H$ at $w$. Note that $\cd_p \left( \Omega_w \right) \leq 1$, 
and so $H^2(\Omega_w, E_{p^{\infty}})=0.$ If $w \nmid p$, then standard arguments (see
for instance \cite[Lemma 3.7]{C}) show that
$$\coker(\gamma_w(k^{\text{cyc}})) = H^2(\Omega_w, E_{p^{\infty}})
=0.$$ If $w|p$, then the main result of Coates-Greenberg \cite{CG} shows
that $$\coker(\gamma_w(k^{\text{cyc}})) = H^2(\Omega_w,
\widetilde{E}_{w, p^{\infty}}) =0.$$ Here, $\widetilde{E}_{w,
p^{\infty}}$ denotes the image under reduction modulo $w$ of
$E_{p^{\infty}}.$ Hence, we find that $\coker(\gamma_w(k^{\text{cyc}})) =0$
for each prime $w$ above $v$ in $k_{\infty}$. It follows that $\gamma_S(k^{\text{cyc}})$ is
surjective. Since the map $\lambda_S(k^{\text{cyc}})$ is also
surjective by $(\ref{locsurj})$, it follows that
$\eta_S(k_{\infty})$ is surjective as required. $\Box$
\end{proof}

\begin{lemma}\label{hsss}
If $E$ has good ordinary reduction at each prime above $p$ in $k$, then for all $i\geq 1$,
$H^i(H, H^1(G_S(k_{\infty}), E_{p^{\infty}})) = 0.$
\end{lemma}

\begin{proof} See \cite[Lemma 2.4]{CSS}. The same proof works here,
using Theorem \ref{cycloc} with the fact that $\cd_p \left( H \right) =1$. $\Box$
\end{proof}

\begin{lemma}\label{H1vanish}
If $E$ has good ordinary reduction at each prime above $p$ in $k$, then 
$H^1(H,\Sel(E/k_{\infty}))=0.$ \end{lemma}

\begin{proof} See \cite[Lemma 2.5]{CSS}.
Let $A_{\infty} = \text{Im}(\lambda_S(k_{\infty})).$ Lemma
$\ref{hsss}$ with $i=1$ gives the exact sequence \begin{align}\label{Ainfty} 0
\longrightarrow \Sel(E/k_{\infty})^H \longrightarrow
H^1(G_S(k_{\infty}), E_{p^{\infty}})^H \longrightarrow A_{\infty}^H
\longrightarrow H^1(H, \Sel(E/k_{\infty})) \longrightarrow 0. \end{align}
Recall that the map $\eta_S(k_{\infty}): H^1(G_S(k_{\infty}), E_{p^{\infty}})^H 
\longrightarrow A_{\infty}^H$ is surjective by Lemma $\ref{Hloc}.$ Now,
\begin{align*} A_{\infty}^{H} = \bigoplus_{v \in S}J_v(k_{\infty})^H, \end{align*} 
and so it follows that $H^1(H, S(E/k_{\infty}))=0$. $\Box$ \end{proof}

\begin{lemma}\label{Jvanish}If $E$ has good ordinary reduction at each prime above $p$ in $k$, 
then \newline $H^1(H, \bigoplus_{v \in S} J_v(k_{\infty}))=0.$ \end{lemma}

\begin{proof} See \cite[Lemma 2.8]{CSS}. The same proof works here, using the
fact that $\cd_p \left( H \right) =1$. $\Box$ \end{proof} We may now deduce the following result.

\begin{theorem} \label{torsion} If $E$ has good ordinary reduction at each prime above $p$ in $k$, 
then $X(E/k_{\infty})$ is $\Lambda(G)$-torsion.\end{theorem}

\begin{proof} See the arguments of \cite[Theorem 2.8, and Corollary 2.9]{HV}, following \cite[Proposition 2.9]{CSS}. 
A standard deduction, as given for instance in \cite[ $\S 2$, Remark 2.5]{HV}, 
reduces the claim to showing the surjectivity of the localization map \begin{align*}
\lambda_S(k_{\infty}): H^1(G_S(k_{\infty}), E_{p^{\infty}}) &\longrightarrow \bigoplus J_v(k_{\infty}).
\end{align*} So, let $A_{\infty} = \im(\lambda_S(k_{\infty})).$ Taking the $H$-cohomology of  the 
exact sequence \begin{align*} 0 \longrightarrow \Sel(E/k_{\infty})\longrightarrow H^1(G_S(k_{\infty}), 
E_{p^{\infty}}) \longrightarrow A_{\infty} \longrightarrow 0, \end{align*} we obtain from Lemma 
\ref{hsss}  the identification \begin{align*} H^1(H, A_{\infty}) = H^2(H, \Sel(E/k_{\infty})).\end{align*} 
Note that since $\cd_p(H)=1$, we have that $H^2(H,\Sel(E/k_{\infty}))=0$, and hence that
$H^1(H, A_{\infty}) =0$. Let $B_{\infty} = \coker(\lambda_S(k_{\infty})).$ By Lemma $(\ref{Jvanish})$,
\begin{align*} H^1(H, \bigoplus_{v \in S}J_v(k_{\infty}))=0. \end{align*}Taking $H$-cohomology of the 
exact sequence \begin{align*}0 \longrightarrow A_{\infty} \longrightarrow \bigoplus_{v \in S}J_v(k_{\infty}) 
\longrightarrow B_{\infty} \longrightarrow 0, \end{align*} we deduce from Lemma $\ref{Hloc}$ that
\begin{align*} B_{\infty}^{H} = H^1(H, A_{\infty}) =0. \end{align*} Since $H$ is pro-$p$,
and $B_{\infty}$ a discrete $p$-primary $H$-module, it follows that
$B_{\infty}$ itself must vanish. Hence $\lambda_S(k_{\infty})$ is surjective. $\Box$
\end{proof} When $X(E/k_{\infty})$ is $\Lambda(G)$-torsion,
the structure theory of torsion $\Lambda(G)$-modules (\cite[Chapter VII,
$\S 4.5$]{BB}) gives a pseudoisomorphism 
\begin{align}\label{2vs} X(E/k_{\infty})
\longrightarrow  \bigoplus_{i=1}^t \Lambda(G)/p^{a_i} \oplus
\bigoplus_{j=1}^u \Lambda(G)/g_j^{b_j}.\end{align} Here, the indices $s$, $t$, $a_i$ and $b_j$ are
all positive integers, and each $g_j$ can be viewed as an irreducible monic
distinguished polynomial $g_j(T_1, T_2)$ (with respect to a fixed isomorphism
$\Lambda(G) \cong {\bf{Z}}_p[[T_1, T_2]]$). The characteristic power series \begin{align*}
\operatorname{char}_{\Lambda(G)} X(E/k_{\infty}) &= \prod_{i=1}^t p^{a_i} \cdot \prod_{j=1}^u g_j^{b_j}
\end{align*} is again well defined up to unit in $\Lambda(G)$. As in the cyclotomic setting, one uses it 
to define the $\Lambda(G)$-module invariants
\begin{align*} \mu_{\Lambda(G)}(X(E/k_{\infty})) = \sum_{i=1}^t a_i ~~~\text{and~~}
\lambda_{\Lambda(G)}\left(X(E/k_{\infty})\right) = \sum_{j=1}^u b_j \cdot
\deg(g_j).\end{align*} \end{remark}

\begin{remark}[The invariant $ \mu_{\Lambda(G)}(X(E/k_{\infty})) $.]

Let us now review what is known about the invariant
$\mu_{\Lambda(G)}(X(E/k_{\infty}))$. Suppose more generally that $G$
is any pro-$p$ group, and $Y$ any finitely-generated torsion
$\Lambda(G)$-module. The structure theory of $\Lambda(G)$ modules
shown in \cite[Chapter VII, $\S 4.5$]{BB}) again gives a
pseudoisomorphism analogous to $(\ref{2vs})$, and so we may define
the associated invariant $\mu_{\Lambda(G)}(Y).$ Let $Y(p)$ denote the submodule
of elements of $Y$ annihilated by some power of $p$. It is well
known (see for instance \cite{H}) that the cohomology groups $H^i(G,
Y)$ are finitely-generated ${\bf{Z}}_p$-modules for all $i \geq 0$,
and hence that the cohomology groups $H^i(G, Y(p))$ are finite for all $i \geq 0.$ 
The invariant $\mu_{\Lambda(G)}(Y)$ is then seen to be given by the
formula \begin{align}\label{muformula}p^{\mu_{\Lambda(G)}(Y)} =
\prod_{i \geq 0} \lvert H_i(G, Y(p))\rvert^{(-1)^i} = \chi(G,
Y(p)),\end{align} where $\chi(G, Y(p))$ is by definition the finite
$G$-Euler characteristic of $Y(p)$. Given $L$ any extension of $k$
contained in $k_{\infty}$, let us write $$\mathfrak{X}(E/L) = X(E/L)/X(E/L)(p).$$ 

\begin{proposition} \label{generalmu}
If $E$ has good ordinary reduction at $p$, and
$\mathfrak{X}(E/k_{\infty})$ is finitely-generated over
$\Lambda(H)$, then $\mu_{\Lambda(G)}(X(E/k_{\infty})) = \mu_E(k)$.
\end{proposition}

\begin{proof}
See \cite[Propostion 2.9]{CSS}, we give a sketch of the proof. 
Note that we have $X(E/k_{\infty})_H =H_0(H, X(E/k_{\infty})).$ Note as well that 
$H_1(H, X(E/k_{\infty})) =0$ by Lemma \ref{H1vanish}. Taking $H$-homology 
of the short exact sequence $$0 \longrightarrow X(E/k_{\infty})(p) \longrightarrow
X(E/k_{\infty}) \longrightarrow \mathfrak{X}(E/k_{\infty})
\longrightarrow 0,$$ we obtain a short exact sequence of
$\Lambda(\Gamma)$-modules
\begin{align*} 0 \longrightarrow H_1(H, \mathfrak{X}(E/k_{\infty}))
&\longrightarrow H_0(H, X(E/k_{\infty})(p))\\ &\longrightarrow
H_0(H, X(E/k_{\infty})) \longrightarrow H_0(H,
\mathfrak{X}(E/k_{\infty})) \longrightarrow 0.\end{align*} Following
\cite[Proposition 1.9]{H}, we then show that the alternating sum of
$\mu_{\Lambda(\Gamma)}$-invariants along this sequence vanishes. Moreover, the
$\mu_{\Lambda(\Gamma)}$-invariants of the two central terms can be
computed as follows. For $H_0(H, X(E/k_{\infty})) =
X(E/k_{\infty})_H,$ it is well known (see the proof of Theorem
\ref{mhg} below for instance) that restriction on cohomology induces
a $\Lambda(\Gamma)$-homomorphism
$$\alpha: X(E/k_{\infty})_H \longrightarrow X(E/k^{\text{cyc}})$$
with $\ker(\alpha)$ finitely-generated over ${\bf{Z}}_p$ and
$\coker(\alpha)$ finite. We deduce that
$$\mu_{\Lambda(\Gamma)}((X(E/k_{\infty})_H) =
\mu_{\Lambda(\Gamma)}((X(E/k^{\text{cyc}})) = \mu_E(k).$$ For
$H_0(H, X(E/k_{\infty})(p)),$ consider the Hochschild-Serre spectral
sequence \begin{align*} 0 \longrightarrow H_0(\Gamma, H_i(H,
X(E/k_{\infty})(p)) &\longrightarrow H_i(G, X(E/k_{\infty})(p))\\
&\longrightarrow H_1(\Gamma, H_{i-1}(H, X(E/k_{\infty})(p))
\longrightarrow 0. \end{align*} We deduce that $$\chi(G,
X(E/k_{\infty})(p)) = \prod_{i=0}^{1} \chi(\Gamma, H_i(H,
X(E/k_{\infty})(p)))^{(-1)^i},$$ and so
$$\mu_{\Lambda(G)}(X(E/k_{\infty})) = \sum_{i=0}^{1}(-1)^i
\mu_{\Lambda(\Gamma)}(H_i(H, X(E/k_{\infty})(p))).$$ Putting terms
together from the first (alternating sum) sequence above, we find
that $\mu_{\Lambda(G)}(X(E/k_{\infty})) =$  $$\mu_E(k) + \sum_{i
=0}^{1} (-1)^{i+1}\mu_{\Lambda(\Gamma)}(H_i(H,
\mathfrak{X}(E/k_{\infty}))) + \sum_{i =0}^{1}
(-1)^{i}\mu_{\Lambda(\Gamma)}(H_i(H, X(E/k_{\infty})(p))).$$ Recall
that $H_i(H, X(E/k_{\infty}))=0$ for all $i \geq 0$ by Lemma
$\ref{H1vanish}$. Taking $H$-cohomology of the short exact sequence
$$0 \longrightarrow X(E/k_{\infty})(p) \longrightarrow
X(E/k_{\infty}) \longrightarrow \mathfrak{X}(E/k_{\infty})
\longrightarrow 0,$$ obtain that $H_1(H, X(E/k_{\infty})(p))) =
H_2(H, \mathfrak{X}(E/k_{\infty}))=0.$ Deduce that
$$\mu_{\Lambda(G)}(X(E/k_{\infty})) =\mu_E(k) + \sum_{i =0}^{1}
(-1)^{i+1}\mu_{\Lambda(\Gamma)}(H_i(H,
\mathfrak{X}(E/k_{\infty}))).$$ Since we assume that
$\mathfrak{X}(E/k_{\infty})$ is finitely-generated over
$\Lambda(H),$ it follows that $\mathfrak{X}(E/k_{\infty})_H$ is
finitely-generated over ${\bf{Z}}_p$. Thus,
$$\mu_{\Lambda(\Gamma)}(H_i(H, \mathfrak{X}(E/k_{\infty})))=0.$$ In
particular, $\mu_{\Lambda(G)}(X(E/k_{\infty})) = \mu_E(k)$ as
claimed. $\Box$ \end{proof} \end{remark}

\begin{remark}[The $G$-Euler characteristic of $\Sel(E/k_{\infty})$.] We 
now give a formula for the $G$-Euler characteristic of $\Sel(E/k_{\infty})$,
\begin{align*} \chi(G, \Sel(E/k_{\infty})) &= \prod_{i \geq 0} \vert H^i(G, 
\Sel(E/k_{\infty})) \vert^{(-1)^ i}, \end{align*} which in the setup described 
above is well defined (i.e. finite). Note that this invariant is related to the 
characteristic power series $\operatorname{char}_{\Lambda(G)} X(E/k_{\infty}) $
by the formula \begin{align*}\chi(G, \Sel(E/k_{\infty})) &= \vert 
\operatorname{char}_{\Lambda(G)} X(E/k_{\infty})(0) \vert_p^{-1}, \end{align*}
where $\operatorname{char}_{\Lambda(G)} X(E/k_{\infty})(0)$ denotes the image
of $\operatorname{char}_{\Lambda(G)} X(E/k_{\infty}) $ under the natural augmentation 
map $\Lambda(G) \longrightarrow {\bf{Z}}_p$. We must first establish the following result.

\begin{lemma}\label{finite} If $E$ has good ordinary reduction at each prime above $p$ in $k$,
then the $p$-primary torsion subgroup $E(k_{\infty})_{p^{\infty}}$ is finite. \end{lemma}

\begin{proof} See the argument of \cite{HV}[Lemma 3.12]. We present the following 
alternative proof. Fix a rational prime $v$ that remains inert in $k$ and does not equal
$p$. Write $k_v$ to denote the localization of $k$ at the prime above $v$. Write $k_v^{\cyc}$
to denote the cyclotomic ${\bf{Z}}_p$-extension of $k_v$. By Imai's theorem \cite{Im} (cf. \cite{CS2}[A.2.7]),
the $p$-primary subgroup of $E(k_v^{\cyc})$ is finite. On the other hand, the prime above $v$ in $k$
splits completely in $D_{\infty}$ by class field theory. Hence, writing $D_{\infty, w}$ to denote the union 
of all completions of $D_{\infty}$ at primes above $v$, we have an isomorphism of local fields 
$D_{\infty, w} \cong k_v$. This induces an isomorphism of Mordell-Weil groups $E(D_{\infty, w}) \cong
E(k_v)$. Hence, writing $k_{\infty, \mathfrak{w}}$ to denote the union of all completions of $k_{\infty}$ at 
primes above $v$, we have the identifications \begin{align*}
E(k_{\infty, \mathfrak{w}}) \cong E(D_{\infty. w} \cdot k_v^{\cyc}) \cong E(k_v^{\cyc}).
\end{align*} Hence, the $p$-primary part of $E(k_{\infty, \mathfrak{w}})$ is seen to be finite by Imai's theorem.
Since $E(k_{\infty})_{p^{\infty}}$ injects into the $p$-primary part of $E(k_{\infty, \mathfrak{w}})$, the result follows. 
$\Box$ \end{proof}

\begin{theorem}\label{Euler} Assume that  $E$ has good ordinary reduction at all primes above $p$ in $k$, that $p \geq 5$,
and that $\Sel(E/k)$ is finite. Then, the $G$-Euler characteristic of $\Sel(E/k_{\infty})$ is well defined, and given 
by the formula

\begin{align*}
\chi(G, \Sel(E/k_{\infty})) &= \frac{\vert \Sha(E/k)(p) \vert}{\vert E(k)_{p^{\infty}} \vert^2} \cdot \prod_{v \mid p} \vert  
\widetilde{E}_v(\kappa_v)(p)\vert^2 \cdot \prod_{v} \vert c_v \vert_p^{-1}.
\end{align*} Here, $\Sha(E/k)(p)$ denotes the $p$-primary part of $\Sha(E/K)$, 
$E(k)_{p^{\infty}}$ the $p$-primary part of $E(k)$, $\kappa_v$ the residue field at $v$, $\widetilde{E}_v$ the
reduction of $E$ over $\kappa_v$, and $c_v=[E(k_v):E_0(k_v)]$ the local Tamagawa factor at a prime 
$v \subset \mathcal{O}_k$. \end{theorem}

\begin{proof} See for instance \cite{HV}[Theorem 4.1] The proof is a standard computation, using the facts that 
(i) $X(E/k_{\infty})$ is $\Lambda(G)$-torsion (by Theorem \ref{torsion} above), (ii) $E(k_{\infty})_{p^{\infty}}$ is finite 
(by Lemma \ref{finite} above), and (iii) $p$ is totally ramified in $k_{\infty}$. $\Box$ \end{proof}\end{remark}

\begin{remark}[$\Lambda(H)$-module structure.]

Let us assume now that $\mu_E(k)=0$. We obtain the following
$\Lambda(H)$-module structure theory for $X(E/k_{\infty})$.

\begin{theorem}
\label{mhg} Suppose that $E$ has good ordinary reduction at $p$,
with $\mu_E(k)=0$. Then, there is a $\Lambda(H)$-module isomorphism
$X(E/k_{\infty})\cong \Lambda(H)^{\lambda_E(k)}.$
\end{theorem}

\begin{proof} By Nakayama's lemma, $X(E/k_{\infty})$
is finitely generated over $\Lambda(H)$ if and only if
$X(E/k_{\infty})_H$ is finitely generated over ${\bf{Z}}_p$, hence by
duality if and only if $S(E/k_{\infty})^H$ is co-finitely generated
over ${\bf{Z}}_p$. Given $n \geq 0$ an integer, let $D_n$ denote 
the degree-$p^n$ extension of $k$ contained in $D_{\infty}$, with 
$D_{n}^{\cyc}$ its cyclotomic ${\bf{Z}}_p$-extension. Let $H_n =
\Gal(k_{\infty}/D_{n}^{\cyc})$. Note that $\cd_p(H_n) \leq 1$. Consider the diagram
$$\begin{CD}
 0 @>>> S(E/k_{\infty})^{H_n} @>>> H^1(G_S(k_{\infty}),
E_{p^{\infty}})^{H_n} @>>>
\bigoplus_{v \in S}J_v(k_{\infty})^{H_n} \\
@.        @AA{\alpha_n}A              @AA{\beta_n}A
@AA{\gamma_n}A     @. \\
0 @>>> S(E/D_{n}^{\text{cyc}}) @>>> H^1(G_S(D_{n}^{\text{cyc}}),
E_{p^{\infty}})
@>>> \bigoplus_{v \in S}J_v(D_{n}^{\text{cyc}}).\\
\end{CD}$$ Here, the horizontal rows are exact sequences, and the 
vertical maps are induced by restriction on cohomology.
We have by inflation-restriction that $\coker(\beta_n) \cong H^2(H_n,
E_{p^{\infty}})=0$ and that $\ker(\beta_n) \cong H^1(H_n,
E_{p^{\infty}}).$ Note that $H^1(H_n, E_{p^{\infty}})$ has
cardinality equal to that of $E(D_{n}^{\text{cyc}})_{p^{\infty}},$
which is finite by Imai's theorem \cite{Im}. Given $v \in S$, fix
a place $w$ above $v$ in $k_{\infty}$. We can then write the 
local restriction map as $$ \gamma_n = \bigoplus_w \gamma_{n, w},$$
where the direct sum ranges over the primes above each $v \in S$
in $D_n$. Let $\Omega_{n, w}$ denote the decomposition group of $H_n$ at $w$. 
We argue as in the proof of Lemma $\ref{cycloc}$ that $\coker(\gamma_n) =0.$ 
Following \cite[Lemma 3.7]{C} also find that $$\coker(\gamma_{n,w}) \cong
H^2(\Omega_{n,w}, E_{p^{\infty}}) = 0 ~\text{~and }
\ker(\gamma_{n,w}) \cong H^1(\Omega_{n,w}, E_{p^{\infty}}).$$ In
particular, since the latter group is known to be finite, it follows
that $\ker(\gamma_n)= \bigoplus_w \ker(\gamma_{n,w})$ is finite. It
then follows from the snake lemma that $\ker(\alpha_n)$ and
$\coker(\alpha_n)$ must be finite. Now, recall that $X(E/k^{\cyc})$
is $\Lambda(\Gamma)$-torsion by Theorem \ref{kato-rohrlich}. Matsuno's 
theorem $\cite{M}$ then implies that $X(E/k^{\cyc})$ has no nontrivial finite
$\Lambda(\Gamma)$-submodule. On the other hand, since $\mu_E(k)=0$,
Hachimori and Matsuno's analogue of Kida's formula $\cite{HM}$
implies that $X(E/D_{n}^{\cyc})$ is $\Lambda(\Gamma_n)$-torsion with
$\Gamma_n = \Gal(D_{n}^{\cyc}/D_n)$ and cyclotomic Iwasawa
invariants $$\lambda_E(D_n) = [D_n:k]\cdot\lambda_E(k) ~\text{ and }~
\mu_E(D_n)=\mu_E(k)=0.$$ Since $D_n$ is not totally real, it follows
from Proposition 7.5 of Matsuno \cite{M} that $X(E/D_{n}^{\cyc})$
has no nontrivial finite $\Lambda(\Gamma_n)$-submodule. In
particular, since $\mu_E(D_n)=0$ for each $n \geq 0$, Matsuno's theorem implies 
that $X(E/D_{n}^{\cyc})$ has no nontrivial
finite ${\bf{Z}}_p$-submodule for any $n \geq 0$. This makes the
inverse limit $X(E/k_{\infty}) = \ilim n X(E/D_{n}^{\cyc})$
${\bf{Z}}_p$-torsionfree, from which it follows that $\ker(\alpha_n)
= \coker(\alpha_n) = 0$ for any $n \geq 0$. Thus, we find an isomorphism of 
${\bf{Z}}_p$-modules $\alpha_0: X(E/k_{\infty})_{H_0} \cong
X(E/k^{\cyc}).$ Let us now put $r = \lambda_E(k).$ Let $x_1, \ldots, x_r$
denote a lift to $X(E/k_{\infty})$ of a fixed ${\bf{Z}}_p$-basis of
$X(E/k_{\infty})_H$. Let $I(H)$ denote the augmentation ideal of $H$
in $\Lambda(H)$. Note that $X(E/k_{\infty})_H =
X(E/k_{\infty})/I(H).$ Let $Y$ denote the $\Lambda(H)$-submodule of
$X(E/k_{\infty})$ generated by $x_1, \ldots x_r$. Observe that
$$I(H)(X(E/k_{\infty})/Y) = (I(H)X(E/k_{\infty})+Y)/Y = X(E/k_{\infty})/Y,$$ and so
$X(E/k_{\infty})=Y$ by Nakayama's lemma. In particular, this gives
an isomorphism of $\Lambda(H)$-modules
$$X(E/k_{\infty})\cong \Lambda(H)^r, ~~~~~ \sum_i a_ix_i
\longmapsto \sum_i a_i e_i,$$ where $e_1, \ldots, e_r$ is a standard
$\Lambda(H)$-basis of $\Lambda(H)^r.$ Observe now that
$X(E/k_{\infty})$ has no nontrivial finite $\Lambda(H)$-submodule,
thus making it $\Lambda(H)$-torsionfree. $\Box$
\end{proof}

\begin{corollary}\label{2mu}
Suppose that $E$ has good ordinary reduction at each prime above $p$ in $k$, with
$\mu_E(k)=0$. Then, $\mu_{\Lambda(G)}\left( X(E/k_{\infty})\right) = \mu_E(k) = 0$.\end{corollary}

\begin{proof}
The result follows from argument of Theorem \ref{mhg} above, namely by
using Matsuno's theorem \cite{M} and the main result of
Hachimori-Matsuno \cite{HM} to deduce that $X(E/k_{\infty})$ is
$\Lambda(H)$-torsionfree. $\Box$ \end{proof} We also deduce
from Theorem \ref{mhg} the following consequence for the $\Lambda(H)$-corank
of the $p$-primary parts of the Tate-Shafarevich group $\Sha(E/k_{\infty})$. That is, recall that we consider
the short exact descent sequence of $\Lambda(H)$-modules
\begin{align*} 0 \longrightarrow E(k_{\infty})\otimes {\bf{Q}}_p/{\bf{Z}}_p \longrightarrow 
\Sel(E/k_{\infty}) \longrightarrow \Sha(E/k_{\infty})(p) \longrightarrow 0, \end{align*}
as well as the dual exact sequence \begin{align}\label{dualSES}
0 \longrightarrow \Zha(E/k_{\infty}) \longrightarrow X(E/k_{\infty}) \longrightarrow 
\mathcal{E}(E/k_{\infty}) \longrightarrow 0. \end{align} Here, $\Zha(E/k_{\infty})$ is the Pontryagin dual 
of $\Sha(E/k_{\infty})(p)$, and $\mathcal{E}(E/k_{\infty})$ is that of $E(k_{\infty}) \otimes {\bf{Q}}_p /{\bf{Z}}_p$.
Recall that we let $\epsilon(E/k, 1) = \epsilon(f/k, 1)$ denote the root number of the complex $L$-function
$L(E/k, s) = L(f \times \Theta_k, s)$.

\begin{proposition}\label{tsrank} Assume that $p$ is odd, and moreover that $p$ does not divide the class number 
of $k$ if the root number $\epsilon(E/k, 1)$ equals $-1$. If $E$ has good ordinary reduction at each prime above $p$ in $k$ with 
$\mu_E(k)=0$, then \begin{align*}\rk_{\Lambda(H)}\Zha(E/k_{\infty}) = \begin{cases}\lambda_E(k) &\text{if $\epsilon(E/k, 1)=+1$}
\\ \lambda_E(k) -1&\text{if $\epsilon(E/k, 1)=-1.$}\end{cases} \end{align*} \end{proposition}

\begin{proof}
Observe that $(\ref{dualSES})$ is a short exact sequence of finitely generated $\Lambda(H)$-modules. We know
by Theorem \ref{mhg} that the $\Lambda(H)$-rank of $X(E/k_{\infty})$ is $\lambda_E(k)$. On the other 
hand, we claim that \begin{align}\label{MWrank}\rk_{\Lambda(H)}\mathcal{E}(E/k_{\infty}) = \begin{cases}0 &\text{if $\epsilon(E/k, 1)=+1$}\\
1&\text{if $\epsilon(E/k, 1)=-1.$}\end{cases} \end{align} To see why this is so, let $K$ be any finite extension of $k$ contained in $k^{\cyc}$.
A simple exercise shows that $K$ is a totally imaginary quadratic extension of its maximal totally real subfield $F$. Let $D_{\infty}^K$ denote the compositum
extension $KD_{\infty}$, with Galois group $\Omega_K = \Gal(D_{\infty}^K/K)$. We claim that for any such $K$, we have the rank formula 
\begin{align*} \rk_{\Lambda(\Omega_K)}\mathcal{E}(E/D^K_{\infty}) = \begin{cases}0 &\text{if $\epsilon(E/k, 1)=+1$}
\\1&\text{if $\epsilon(E/k, 1)=-1.$}\end{cases} \end{align*} Indeed, in the first case with $\epsilon(E/k, 1)=+1$, the formula follows from the 
relevant nonvanishing theorem of Cornut-Vatsal \cite[Theorem 1.4]{CV} over $F$ plus the relevant rank theorem(s) of 
Nekovar \cite[Theorem B, Theorem B', and Corollary]{Nek}. In the second case with $\epsilon(E/k, 1)=-1$, the formula follows form the 
relevant nonvanishing theorem of Cornut-Vatsal \cite[Theorem 1.5]{CV} over $F$ plus the relevent rank theorem of Howard \cite[Theorem B (a)]{Ho1}.
Note that to invoke the result of Howard \cite{Ho1} in the latter setting, we have used the classical result due to Iwasawa \cite{Iw} that if $p$
does not divide the class number of $k$, then $p$ does not divide the class number of any finite extension $K$. Taking the inductive limit over all finite 
extensions $K$ of $k$ contained in $k^{\cyc}$, we obtain the stated formula $(\ref{MWrank})$. The result then follows immediately from the exactness of 
$(\ref{dualSES})$. $\Box$ \end{proof} \end{remark} 

\section{Divisibility criteria}

We now discuss various divisibility criteria for the two-variable main conjecture
(Conjecture \ref{2vmc} (iii) above). In particular, granted suitable hypotheses, 
we prove one divisibility of the two-variable main conjecture.

\begin{remark}[Greenberg's criterion.]

The following criterion was suggested to the author by Ralph Greenberg. It
reduces one divisibility of the two-variable main conjecture (Conjecture \ref{2vmc} (iii)) to a certain
specialization criterion for finite order characters of the Galois group $\Gamma = \Gal(k^{cyc}/k)$.
Let us first fix an isomorphism \begin{align}\label{fixedisom} \Lambda(G) \cong {\bf{Z}}_p[[T_1, T_2]], ~
\left( \gamma_1, \gamma_2 \right) &\longmapsto \left( T_1 +1, T_2 +1
\right).\end{align}  Here, we have fixed a topological generator $\gamma_1$ 
of $\Gamma$, as well as a topological generator $\gamma_2$ of
$\Omega$. Fix $f \in S_2(\Gamma_0(N))$ a $p$-ordinary eigenform,
as required for the construction of the $p$-adic $L$-function of Theorem \ref{2vinterpolation}. Recall that we 
write $X(f/k_{\infty})$ to denote the Pontryagin dual of the $p^{\infty}$-Selmer group associated 
to $f$ in $k_{\infty}/k$. If $f$ is the eigenform associated to an elliptic curve $E$ defined over ${\bf{Q}}$, then 
a standard argument allows us to make the identification $X(f/k_{\infty}) =  X(E/k_{\infty})$. In what follows, we shall
fix an elliptic curve $E$ over ${\bf{Q}}$ of conductor $N$ as described in the introduction, with $f$ the eigenform 
associated to $E$ by modularity. We shall then make the identification $X(f/k_{\infty}) = X(E/k_{\infty})$ implicitly in what follows. 

Let $g(T_1, T_2)$ denote the $\Lambda(G)$-characteristic power series of $X(f/k_{\infty})$, or rather
its image under the fixed isomorphism $(\ref{fixedisom})$. (We take this to be zero if $X(f/k_{\infty})$ 
is not $\Lambda(G)$-torsion). Let $L(T_1, T_2) = L_p(f, k)(T_1, T_2)$ denote the image under $(\ref{fixedisom})$ 
of the two-variable $p$-adic $L$-function $L_p(f, k) \in \Lambda(G)$ associated to $f$ by Theorem \ref{2vinterpolation}.
Recall that we write $\Psi$ to denote the set of finite order characters of $\Gamma 
= \Gal(k^{cyc}/k)$. Given an element element $\lambda \in \Lambda(G) $ with associated power series 
$\lambda(T_1, T_2) \in {\bf{Z}}_p[[T_1, T_2]]$, we can and will invoke the usual Weierstrass preparation
theorem for $\lambda(T_1, T_2)$ as an element of the one-variable power series ring $R[[T_1]]$ with
$R= {\bf{Z}}_p[[T_2]]$. We refer the reader to the discussion in Venjakob \cite[Example 2.4, Theorem 3.1, and Corollary 3.2]{Ven}
for a more general account of the situation.

\begin{theorem}\label{greenberg}
Suppose that $p$ does not divide the specialization $g(T_1,0)$. Assume that for each character $\psi
\in \Psi$, we have the inclusion of ideals 
\begin{align} \label{harddiv} \left( L(T_1, \psi(T_2)) \right) \subseteq
\left( g(T_1,\psi(T_2)) \right) \text{  in $\mathcal{O}_{\psi}[[T_1, T_2]]$.} \end{align} Then, we have the inclusion of ideals
\begin{align*} \left( L(T_1, T_2) \right) \subseteq
\left( g(T_1,T_2) \right) \text{  in ${\bf{Z}}_p[[T_1, T_2]]$.} \end{align*} 
\end{theorem}

\begin{proof} Observe that we may write $$g(T_1,T_2) =
\sum_{i=0}^{\infty}a_i(T_2) \cdot T_1^i,$$ with $a_i(T_2) \in
{\bf{Z}}_p[[T_2]]$. Since we assume that $p \nmid g(T_1,0)$, it follows
that for some minimal positive integer $m$, $$g(T_1,0) = \sum_{i=0}^{m} 
a_i(0) \cdot T_1^i,$$ with $a_i(0) \in {\bf{Z}}_p^{\times}$. We claim it then
follows that $$L(T_1,T_2) = h(T_1,T_2) \cdot g(T_1,T_2) + r(T_1,T_2),$$ with $h(T_1,T_2)$ a
polynomial in ${\bf{Z}}_p[[T_1, T_2]]$, and $r(T_1,T_2)$ a remainder polynomial in ${\bf{Z}}_p[[T_2]]$
of degree less than $m$. Now, the remainder term is given by $$r(T_1,T_2) =
\sum_{j=0}^{m-1} c_j(T_2) \cdot T_1^j,$$ with $c_j(T_2) \in
{\bf{Z}}_p[[T_2]].$ Granted the inclusion $(\ref{harddiv})$ for each $\psi \in \Psi$, we have that
$$r(T_1, \psi(T_2)) = 0$$ for each $\psi \in \Psi$. It then follows from the Weierstrass preparation theorem that
$$c_j(\psi(T_2)) = 0$$ for each $\psi \in \Psi$ and $j \in \lbrace 0,
\ldots, m-1 \rbrace.$ Hence, we conclude that $r(T_1,T_2)=0$. 
$\Box$ \end{proof} We obtain the following immediate consequence.

\begin{corollary}\label{gpw} Assume Hypothesis \ref{XXX} (i) and (ii). Suppose that for each character $\psi
\in \Psi$, we have the inclusion of ideals 
\begin{align} \label{harddiv} \left( L(T_1, \psi(T_2)) \right) \subseteq
\left( g(T_1,\psi(T_2)) \right) \text{  in $\mathcal{O}_{\psi}[[T_1, T_2]]$.} \end{align} Then, we have the inclusion of ideals
\begin{align*} \left( L(T_1, T_2) \right) \subseteq
\left( g(T_1,T_2) \right) \text{  in ${\bf{Z}}_p[[T_1, T_2]]$.} \end{align*} \end{corollary}
\begin{proof} Theorem \ref{greenberg} requires that $p$ does not divide the specialization of the characteristic power series 
$g(T_1, 0)$, equivalently that the dihedral or anticyclotomic $\mu$-invariant associated to $f$ in the tower $D_{\infty}/k$
vanishes. Assuming Hypothesis \ref{XXX} (i) and (ii), the main result of Pollack-Weston \cite{PW} shows that this is always 
the case if the underlying eigenform $f$ is $p$-ordinary. $\Box$ \end{proof} \end{remark}

\begin{remark}[A basechange criterion]

Let $K$ be any finite extension of $k$ contained in the cyclotomic extension $k^{\cyc}$. Let $D^K_{\infty}$ denote the 
compositum extension $K D_{\infty},$ with $\Omega_K = \Gal(D^K_{\infty}/K)$ the corresponding Galois group. Note that $\Omega_K$
is topologically isomorphic to ${\bf{Z}}_p.$ Let $\Psi_K$ denote the set of (primitive) characters of order $[K:k]$ of the Galois group $\Gal(K/k)$. 
Hence, we have the decomposition \begin{align*} \Psi = \bigcup_{k \subset K \subset k^{\cyc}} \Psi_K. \end{align*}  Recall that given a character 
$\psi \in \Psi,$ we write $\mathcal{O}_{\psi}$ to denote the ring of integers obtained from ${\bf{Z}}_p$ by adding the values of $\psi.$ 
 Let us also write $\mathcal{O}_{\Psi_K}$ to denote the ring of integers obtained by adding to ${\bf{Z}}_p$ the values 
 of each of the characters in the set $\Psi_K$.  Given a polynomial $f(T_1, T_2) \in {\bf{Z}}_p[[T_1, T_2]]$, let us write 
 \begin{align}\label{prodspec} f(T_1, T_2^K) = \prod_{\psi \in \Psi_K} f(T_1, \psi(T_2))\end{align} to denote the product of 
 specializations of $f(T_1, T_2)$ to the characters of the set $\Psi_K.$ Note that this specialization product $f(T_1, T_2^K)$ lies
 in the polynomial ring ${\bf{Z}}_p[[T_1, T_2^K]] = \mathcal{O}_{\Psi_K}[[T_1]]$. Note as well that we have the identifications  \begin{align*} f(T_1, T_2^k)
 = f(T_1, {\bf{1}}(T_2)) = f(T_1, 0) \in {\bf{Z}}_p[[T_1]]. \end{align*}

\begin{proposition}\label{bccrit} Assume that for any finite extension $K$ of $k$ contained in
$k^{\cyc}$, we have the inclusion of ideals \begin{align}\label{divK} \left(   L(T_1, T_2^K) \right) 
\subseteq  \left(   g (T_1, T_2^K) \right) \text{  in  }  \mathcal{O}_{\Psi_K}[[T_1]].
\end{align} Assume additionally that the root number of the central value $L(f/ k, 1)$ is $+1$, and moreover 
that we have a nontrivial equality of ideals for $K=k,$  
\begin{align}\label{eqk}\left(   L(T_1, 0) \right) = \left( g(T_1,0) \right)  \text{ in }  {\bf{Z}}_p[[T_1]]. \end{align}
Then, for each character $\psi \in \Psi$, we have the inclusion of ideals \begin{align*}
\left( L(T_1, \psi(T_2)) \right) \subseteq \left( g(T_1, \psi(T_2))\right) \text{ in } \mathcal{O}_{\psi}[[T_1]]. \end{align*}\end{proposition}

\begin{proof}  Since we assume that the root number $\epsilon(f/k,1)$ is equal to $+1$, we know for instance by 
the nonvanishing theorems of  Vatsal \cite{Va} and more generally Cornut-Vatsal \cite{CV} that the $p$-adic $L$-function $L(T_1, 0)$ does 
not vanish identically. Let $K$ be any finite extension of $k$ contained in $k^{\cyc}$. Using the equality $(\ref{eqk})$, we may then divide each side of $(\ref{divK})$ 
by the corresponding ideals in $(\ref{eqk})$ to obtain for each extension $K$ the inclusion of ideals 
\begin{align}\label{quotient}  \left(  \frac{L(T_1, T_2^K)}{L(T_1, 0)}   \right) \subseteq   \left(  \frac{g(T_1, T_2^K)}{g(T_1, 0)}   \right) 
\text{ in } \mathcal{O}_{\Psi_K}[[T_1]]. \end{align}  Now, the divisibility $(\ref{divK})$ implies that
we have for each extension $K$ the relation
$$g(T_1, T_2^K) = f(T_1, T_2^K) \cdot L(T_1, T_2^K) + r(T_1, T_2^K).$$ Here, $f(T_1,
T_2^K)$ denotes some polynomial in ${\bf{Z}}_p[[T_1, T_2^K]] = \mathcal{O}_{\Psi_K}[[T_1]]$, and $r(T_1, T_2^K)$ the corresponding 
remainder term. It then follows from $(\ref{quotient})$ that \begin{align*} \prod_{\psi \in \Psi_K \atop \psi \neq 1} r(T_1, \chi(T_2)) =0.\end{align*} 
Hence, we deduce that  for each finite extension $K$ of $k$ contained in $k^{\cyc},$ 
there exists a nontrivial character $\psi \in \Psi_K$ such that \begin{align}\label{twisted}
\left( L(T_1, \psi(T_2)) \right) \subseteq \left( g(T_1, \psi(T_2))\right) \text{ in } \mathcal{O}_{\psi}[[T_1]]. \end{align}
We now argue that if the divisibility $(\ref{twisted})$ holds for one (nontrivial) character in $\Psi_K$, then it holds for 
all (nontrivial) characters in $\Psi_K$. To see why this is, let $\mathcal{L}(E/k, \mathcal{W}, 1) = \mathcal{L}(f \times \Theta(\mathcal{W}), 1)$ 
denote the value \begin{align}\label{algL} \frac{L(f \times \Theta(\mathcal{W}), 1)}{8\pi \langle f, f \rangle}, \end{align} where $\mathcal{W}$
is any finite order character of the Galois group $G$. Recall that the value $(\ref{algL})$ is algebraic by Shimura's theorem \cite{Sh1}. In
particular, for any finite order character $\rho$ of $\Omega$, the values $\mathcal{L}(f \times \Theta(\rho\psi), 1)$ with $\psi \in \Psi_K$ are
Galois conjugate by Shimura's theorem. Hence, by uniqueness of interpolation series, we deduce that the specializations $L(T_1, \psi(T_2))$
with $\psi \in \Psi_K$ are Galois conjugate. We can then deduce that if the divisibility $(\ref{twisted})$ holds for one character $\psi \in \Psi_K$,
then it holds for all characters $\psi \in \Psi_K$. Taking the union of all finite extensions $K$ of $k$ contained in $k^{\cyc}$, the result follows.
$\Box$ \end{proof}

\begin{corollary} Keep the hypotheses of Proposition \ref{bccrit} above. If $p$ does not divide the specialization $g(T_1, 0)$, then there is an 
inclusion of ideals \begin{align} \label{harddiv} \left( L(T_1, \psi(T_2)) \right) \subseteq
\left( g(T_1,\psi(T_2)) \right) \text{  in $\mathcal{O}_{\psi}[[T_1, T_2]]$.} \end{align} 
\end{corollary}

\begin{proof} Apply Theorem \ref{greenberg} to Proposition \ref{bccrit} above. $\Box$ \end{proof}

\begin{remark}[Some remarks on further reductions.]

A simple argument shows that each finite extension $K$ of $k$ contained in $k^{\cyc}$ is a totally imaginary quadratic extension of 
its maximal totally real subfield $F$. Each such totally real field $F$ is abelian. Hence, we can associate to $f$ a Hilbert modular 
eigenform ${\bf{f}}$ over $F$ via the theory of cyclic basechange. It is then simple to see (via Artin formalism for instance) that the root
number of the complex Rankin-Selberg $L$-function $L({\bf{f}} \times \Theta_K, s)$ is equal to that of $L(E/k, s) = L(f \times \Theta_k, s)$.
In particular, the divisibilities $(\ref{divK})$ of Proposition \ref{bccrit} would follow from the dihedral/anticyclotomic main conjectures for 
${\bf{f}}$ in the dihedral/anticyclotomic ${\bf{Z}}_p^d$-extension of $K$, where $d = [F:{\bf{Q}}]$. For results in this direction, see for instance
 the generalizations to totally real fields of work of Bertolini-Darmon \cite{BD} (as well as Pollack-Weston \cite{PW}) by Longo \cite{Lo3}
and the author \cite{VO2}. For the equality condition $(\ref{eqk})$ of Proposition \ref{bccrit}, see the result of 
Howard \cite[Theorem 3.2.3]{Ho} with the main result of Pollack-Weston \cite{PW}. These works combined show that the inclusion
$(L(T_1, 0)) \subseteq (g(T_1, 0))$ often holds, in which case the reverse inclusion $(g(T_1, 0)) \subseteq (L(T_1, 0))$ can be reduced
by Howard \cite[Theorem 3.2.3(c)]{Ho} to a certain nonvanishing criterion for the associated $p$-adic $L$-functions $L(T_1, 0) \in \Lambda(\Omega)$.
\end{remark}

\begin{remark}[Some remarks on the setting of root number minus one.] 

In the setting where the root number $\epsilon(f/k, 1)$ of $L(f/k, 1)$ is equal to $-1$, then we know that $L(T_1, T_2) =0$ by the functional equation for 
$L(T_1, T_2)$ given in Corollary \ref{FE} (derived from the fact that the complex central value $L(f/k, 1)$ vanishes). 
It follows that $L(T_1, T_2^K) =0$ for all finite extensions $K$ of $k$ contained in $k^{\cyc}$. Hence in this setting, the hypotheses
of Proposition \ref{bccrit} do not hold. Indeed, consider the basechange setup described in the remark above, where ${\bf{f}}$
is the basechange Hilbert modular eigenform defined over the maximal totally real subfield $F$ of $K$. The formulation of the analogous 
dihedral/anticyclotomic main conjecture in this setting asserts that each dual Selmer group $X({\bf{f}}/D_{\infty}^K)$ has 
$\Lambda(\Omega_K)$-rank one, and moreover that there is an equality of ideals \begin{align*}
\left(  \operatorname{char}_{\Lambda(\Omega_K)} (X({\bf{f}}/D_{\infty}^K)_{\tors} )  \right)  &= 
\left(  \operatorname{char}_{\Lambda(\Omega_K)} (\mathfrak{X}({\bf{f}}/D_{\infty}^K))  \right)  \text{in $\Lambda(\Omega_K)$}
 \end{align*} Here, $X({\bf{f}}/D_{\infty}^K)_{\tors} $ denotes the $\Lambda(\Omega_K)$-torsion submodule of $X({\bf{f}}/D_{\infty}^K)$,
 and $\mathfrak{X}({\bf{f}}/D_{\infty}^K)$ is the $\Lambda(\Omega_K)$-torsion submodule defined by $\mathfrak{S}({\bf{f}}/D_{\infty}^K)/
 H({\bf{f}}/D_{\infty}^K)$, where $\mathfrak{S}({\bf{f}}/D_{\infty}^K)$ is the compactified Selmer group of ${\bf{f}}$ over $D_{\infty}^K$,
 and $H({\bf{f}}/D_{\infty}^K)$ is the so-called Heegner submodule generated by CM points (defined on an associated quaternionic Shimura curve).
 We refer the reader to Howard \cite[Theorem B]{Ho1} or Perrin-Riou \cite{PR} for more details on this formulation.
Anyhow, the dual Selmer group $X({\bf{f}}/D^K_{\infty})$ does not have a $\Lambda(\Omega_K)$ characteristic power series in this setting. If we adopt the standard convention of taking 
 the characteristic power series to be $0$ in this case, then we obtain for each extension $K$ the trivial equality of ideals $\left(   L(T_1, T_2^K) \right) =  \left(   g (T_1, T_2^K) \right) \text{  in  }  
 \mathcal{O}_{\Psi_K}[[T_1]]$. It therefore seems unlikely that we can do any better than Theorem \ref{gpw} for determining a two-variable divisibility criterion by considering main conjecture divisibilities 
 via basechange. This is especially apparent after noting of the shape of the two-variable main conjecture in this case, as described for instance in Howard \cite{Ho0}. To be somewhat more precise, recall that we fixed a topological 
 generator $\gamma_2$ of $\Gamma$ for our fixed isomorphism $(\ref{fixedisom})$. The two-variable $p$-adic $L$-function $L_p(f, k_{\infty})$ can then be written as a power series \begin{align*}  \mathcal{L}_f = 
 \mathcal{L}_{f, 0} + \mathcal{L}_{f, 1} \cdot (\gamma_2 -1) + \ldots \in \Lambda(G), \end{align*} with coefficients $\mathcal{L}_{f, n} \in {\bf{Z}}_p[[\Omega]]$. In the case where the root number $\epsilon(f/k, 1)$ is $-1$, we know
 by the associated functional equation(s) that $\mathcal{L}_{f,0} =0$. Another result of Howard (proving one divisibility of a conjecture made by Perrin-Riou in \cite{PR}) shows that the second
 term $\mathcal{L}_{f,1}$ can be expressed  as a certain twisted sum of images under any appropriate $p$-adic height pairing of some associated regularized Heegner points (see  \cite[Theorem A]{Ho0}). If $p$ does not 
 divide the level $N$ of $f$, then we know by Theorem \ref{torsion} that $\operatorname{char}_{\Lambda(G)}X(f/k_{\infty})$ exists, equivalently that $g(T_1, T_2) \neq 0.$ Now, two-variable characteristic 
 power series $\operatorname{char}_{\Lambda(G)}X(f/k_{\infty})$ can be written as a power series \begin{align*}  \mathcal{G}_f = \mathcal{G}_{f, 0} + \mathcal{G}_{f, 1} \cdot (\gamma_2 -1) + \ldots \in \Lambda(G), \end{align*} with coefficients 
 $\mathcal{G}_{f, n} \in {\bf{Z}}_p[[\Omega]]$. Hence, if we know that $g(T_1, 0) =0$, then we find that $\mathcal{G}_{f,0}=0.$ This would reduce our task to showing $\mathcal{G}_f \mid \mathcal{L}_f$ in $\Lambda(G)$, where both
 $\mathcal{G}_f$ and $\mathcal{L}_f$ correspond under the fixed isomorphism $(\ref{fixedisom})$ to power series that vanish at $T_2 =0.$ It is then apparent from this fact that comparing the products of specializations to characters
 $\psi \in \Psi_K$ of these power series alone will not give much more information, as $\Psi_K$ contains the trivial character. \end{remark} \end{remark}

\end{document}